\documentclass[reqno,11pt]{amsart}
\usepackage{relsize}
\usepackage{scalerel}
\usepackage[margin=1in]{geometry}
\usepackage{stackengine,wasysym}
\usepackage{todonotes}
\usepackage{comment}
\usepackage{xcolor}
\usepackage{mathrsfs}
\usepackage{dsfont}
\usepackage{mathtools}
\mathtoolsset{showonlyrefs}
\usepackage{enumerate}
\usepackage[hyperfootnotes=false,colorlinks=true,linkcolor = blue,
urlcolor  = blue, citecolor = blue]{hyperref}
\usepackage[sort]{cite}
\usepackage{float}
\usepackage{amsmath, amsthm, amssymb}
\usepackage{times}
\usepackage{color}

\numberwithin{equation}{section}

\allowdisplaybreaks
\usepackage{hyperref}

\newcommand{\pa}{\partial}
\newcommand{\la}{\label}
\newcommand{\fr}{\frac}
\newcommand{\na}{\nabla}
\newcommand{\be}{\begin{equation}}
\newcommand{\ee}{\end{equation}}
\newcommand{\ba}{\begin{array}{l}}
\newcommand{\ea}{\end{array}}

\DeclareMathOperator*{\esssup}{ess\,sup}

\newtheorem{thm}{Theorem}[section]
\newtheorem{prop}[thm]{Proposition}

\newtheorem{cor}[thm]{Corollary}

\newtheorem{rem}{Remark}

\newcommand{\beg}{\begin}

\usepackage{MnSymbol,wasysym}
\newcommand{\N}{\mathbb N}

\newcommand{\R}{\mathbb R}

\def\RR{{\mathbb R}}
\def\TT{{\mathbb T}}
\def\NN{\mathbb N}

\def\ek{{\eta_k}}

\title[3D Nernst-Planck-Boussinesq system]{On the three-dimensional Nernst-Planck-Boussinesq System}

\author[E. Abdo]{Elie Abdo}
\address[E. Abdo]
{	Department of Mathematics \\
     University of California  \\
	Santa Barbara, CA 93106, USA.} \email{elieabdo@ucsb.edu}

\author[R. Hu]{Ruimeng Hu}
\address[R. Hu]
{Department of Mathematics \\
Department of Statistics and Applied Probability \\
     University of California  \\
	Santa Barbara, CA 93106, USA.} \email{rhu@ucsb.edu}

\author[Q. Lin]{Quyuan Lin}
\address[Q. Lin]
{	School of Mathematical and Statistical Sciences \\
Clemson University\\
Clemson, SC 29634, USA.} \email{quyuanl@clemson.edu}
\date{May 3, 2024}

\begin{document}

\begin{abstract}
     In this paper, we analyze a three-dimensional Nernst-Planck-Boussinesq (NPB) system that describes ionic electrodiffusion in an incompressible viscous fluid. This new model incorporates variational temperature and is forced by buoyancy force stemming from temperature and salinity fluctuations, enhancing its generality and realism. The electromigration term in the NPB system displays a complex nonlinear structure influenced by the reciprocal of the temperature that distinguishes its mathematical aspects from other electrodiffusion models studied in the literature. We address the global existence of weak solutions to the NPB system on the three-dimensional torus for large initial data. In addition, we study the long-time dynamics of these weak solutions and the associated relative entropies and establish their exponential decay in time to steady states.
\end{abstract}

\maketitle


\section{Introduction}
Electrodiffusion of ions is a phenomenon that takes place in electrolyte solutions when charged ions are transported in a fluid and undergo the influence of an electric field. Ionic electrodiffusion is prevalent in various real-world applications, including but not limited to neuroscience \cite{jasielec2021electrodiffusion,nicholson2000diffusion,qian1989electro,pods2013electrodiffusion,mori2009numerical,lopreore2008computational,koch2004biophysics,cole1965electrodiffusion,savtchenko2017electrodiffusion}, semiconductor theory \cite{biler1994debye,gajewski1986basic,mock1983analysis}, water purification, desalination, and ion separation \cite{alkhad2022electrochemical,zhu2020ion,lee2018diffusiophoretic,yang2019review,gao2014high,lee2016membrane}, and battery performance and longevity \cite{tan2016computational}. Recently, there has been a resurgence of interest in ionic electrodiffusion, driven by advancements in the ability to precisely control ion transport through selectively charged nanoscale membranes \cite{constantin2019nernst}.

Within a fluid, the dynamics of ions are governed by three fundamental mechanisms: transport by the velocity of the fluid, diffusion by the gradient of the ionic concentrations, and electromigration by the gradient of the electric potential arising from the motion of ions. The electric field generated by the total charge density gives rise to electric forces that govern the fluid flow.  These processes are mathematically described by the Nernst-Planck (NP) equations
\cite{rubinstein1990electro, schmuck2009analysis} that are commonly coupled with various fluid systems depending on the medium, including the Nernst-Planck-Navier-Stokes (NPNS) system in viscous fluids \cite{constantin2019nernst,constantin2021interior,constantin2022existence,abdo2022space,fischer2017global,liu2020global,constantin2021nernst,constantin2022nernst,bothe2014global}, the Nernst-Planck-Euler (NPE) system in ideal fluids \cite{ignatova2021global,abdo2022global,shen2022stability,zhang2015global,zhang2020inviscid}, and the Nernst-Planck-Darcy (NPD) system in porous media \cite{ignatova2022global,abdo2022global}.

Although these models have been extensively studied over the last two decades, there are still many unsolved related mathematical problems, especially under pragmatic physical aspects that have not been taken into account in the literature. A critical constraint in the previously studied models is the exclusion of temperature variations, which simplifies the mathematical challenges that could arise from the nonlinear structure of the model but also limits its real-world application.
In addition, the salinity resulting from the variation of the ionic concentrations creates a buoyancy force that mainly affects the momentum but has also been disregarded in the aforementioned studied models. In order to enhance our physical understanding of such electrodiffusion phenomena, particularly in the context of groundwater and seawater \cite{yang2010numerically,gray1976validity,mayeli2021buoyancy,petcu2009some}, it is imperative to seek further investigations from mathematical and physical points of view. To this end, we analyze a new model called the Nernst-Planck-Boussinesq (NPB) system, where the NP equations and Boussinesq equations are coupled, capturing variations in temperature and accounting for buoyancy forces.

\subsection{Nernst-Planck-Boussinesq system}

Electrodiffusion of $N$ ions is described by the NP equations
\begin{equation}\label{eqn:NP}
    \partial_t c_i + \nabla\cdot j_i = 0, \quad  j_i = uc_i - D_i\nabla c_i - D_i\frac{ez_i}{k_B T} c_i \nabla \Psi, \quad i=1,\dots, N,
\end{equation}
\normalsize
where the nonnegative functions $c_i(x,t)$ denote the $i$-th ionic species concentration, and the flux, denoted by $j_i$, encompasses three components: advection, diffusion, and electromigration. Here $x\in \Omega\subseteq \mathbb R^d$ is the spatial variable and $t\geq 0$ represents the time. The divergence-free function $u(x,t)$ is the velocity field. The positive constants $D_i$, which can vary between different ionic species, denote their respective diffusivities. Additionally, $e$ refers to the elementary charge, the constants $z_i\in \mathbb R$ denote the valences, $k_B$ is Boltzmann's constant, and $T>0$ represents the absolute temperature. The electromigration form can also be written as $-D_i \frac{Fz_i}{RT} c_i \nabla\Psi$, where $F=eN_A$ is the Faraday constant, $R=k_B N_A$ is the molar gas constant, and $N_A$ represents the Avogadro constant.
The electric potential $\Psi$ obeys Gauss's Law (or Maxwell's first equation) and solves a Poisson equation
\begin{equation}\label{eqn:poisson}
    -\varepsilon \Delta \Psi = \rho, \text{ with } \rho =  eN_A\sum\limits_{i=1}^N z_i c_i = F\sum\limits_{i=1}^N z_i c_i,
\end{equation}
\normalsize
where $\varepsilon>0$ is a constant representing the dielectric permittivity of the solvent, and
$\rho$ is the charge density. We refer the reader to \cite{rubinstein1990electro} for a more detailed physical interpretation of the NP equations.

The velocity $u$ and the temperature $T$ obey the Boussinesq approximation: 
\noeqref{eqn:bouss-1}
\noeqref{eqn:bouss-2}
\noeqref{eqn:bouss-3}
\noeqref{eqn:bouss-4}
\noeqref{eqn:bouss-5}
\begin{subequations}\label{sys:bouss}
    \begin{align}
   &\beta_0(\partial_t u + u\cdot \nabla u - \nu\Delta u) + \nabla p = -g\beta \vec{k} - \rho \nabla\Psi, \label{eqn:bouss-1}
   \\
   &\nabla\cdot u = 0,\label{eqn:bouss-2}
   \\
   &\beta=\beta_0(1-\alpha_T(T-T_r) + \alpha_S(S-S_r)),\label{eqn:bouss-3}
   \\
   &\partial_t T + u\cdot \nabla T - \kappa\Delta T = 0,\label{eqn:bouss-4}
   \\
   &S=\sum\limits_{i=1}^N c_i M_i.\label{eqn:bouss-5}
\end{align}
\end{subequations}
Here $S$ is the salinity, $\beta$ is the density, and $T_r$, $S_r$, $\beta_0$ are the average (or reference) values of $T$, $S$, $\beta$, respectively. For simplicity and without loss of generality, we assume $\beta_0=1$. The quantities $\nu>0$ and $\kappa>0$ denote the viscosity and diffusivity, respectively, and $M_i$ is the Molar mass of the $i$-th ionic species. In the Boussinesq approximation, the density is taken to be constant in space and time (homogeneous) except when it is multiplied by $g$, the gravitational acceleration. The nonnegative constants $\alpha_T$ and $\alpha_S$ appearing in \eqref{eqn:bouss-3} are expansion coefficients (see \cite{petcu2009some}). In particular, we assume that the temperature of the whole system is around $T_r$, which makes the Boussinesq approximation a suitable model. In contrast with \cite{petcu2009some}, where the salinity $S$ has its own evolution equation, our current model provides a framework in which $S$ is fully and solely determined and understood from the behavior of the ionic concentrations.

The Boussinesq system \eqref{sys:bouss} coupled with the NP equations \eqref{eqn:NP}--\eqref{eqn:poisson} will be referred to as the NPB system.  In this work, we address in detail the existence of global weak solutions to this model and their long-time behavior on the three-dimensional torus $\TT^3 := \mathbb R^3/\mathbb Z^3$ with periodic boundary conditions.

When $T\equiv T_r$ and $\alpha_S=0$ in the NPB system, corresponding to the case that the variation of temperature and the buoyancy force are ignored, one can combine $g\vec{k}=\nabla(gz)$ with the pressure gradient to recover the NPNS system. In other words, the NPNS system is a special case of the NPB system. A fundamental distinction between the NPB and NPNS systems resides in the nonlinear nature of the temperature effects when taken into consideration. In contrast with the NPB system where the temperature fluctuates and majorly affects the motion of ions, the temperature in the formerly studied NPNS models was assumed to be constant in both space and time, exerting no impact neither on the dynamics of the fluid nor on the evolution of ionic concentrations. 
Another distinction lies in the inclusion of the salinity effects in the buoyancy term, which introduces an additional contribution of the ionic concentrations to the momentum equation \eqref{eqn:bouss-1}.

The NPNS system has been intensely studied in various settings under different boundary conditions. 
The existence of global weak solutions in 3D with blocking boundary conditions has been established in \cite{fischer2017global,jerome2009global}. Global existence for small initial data and forces was proved in \cite{ryham2009existence,schmuck2009analysis}. Under Robin boundary conditions imposed on the electric potential and blocking boundary conditions imposed on the ionic concentrations, the global existence and stability in 2D have been established in \cite{bothe2014global}. 
In \cite{constantin2019nernst}, the authors considered the 2D case under various physical boundary conditions for $c_i$ and demonstrated the existence of smooth global solutions that converge to steady states. The existence of strong solutions in 3D was established in \cite{constantin2021nernst} in the case of two ionic species and many ionic species having equal diffusivities. In \cite{constantin2022nernst}, the nonlinear stability of Boltzmann states was obtained on 2D and 3D bounded domains while instabilities were reported in  \cite{rubinstein2000electro, rubinstein2008direct, zaltzman2007electro}, depending on the boundary conditions obeyed by the potential and concentrations. In the periodic setting, the long-time dynamics were explored in  \cite{abdo2021long,abdo2024long} and the analyticity of  solutions was established in \cite{abdo2022space}.

\subsection{Main challenges and the strategies}

The electromigration term $\na \cdot (D_i\frac{ez_i}{k_B T} c_i \na \Psi)$ in \eqref{eqn:NP} is a source of main challenges in the analysis of the NPB system. An elementary scrutinization of the algebraic structure of this term motivates a non-vanishing condition on the temperature $T$ for the sake of avoiding a vanishing denominator. This issue is tackled by assuming that the initial absolute temperature $T_0(x)$ is bounded from below by a positive constant $T^*$. This initial boundedness criterion is conserved in time, a fact that follows by adapting the approach in \cite{constantin2019nernst} (employed to guarantee the nonnegativity of the ionic concentrations) to our current model. Such an assumption is consistent with the Third Law of Thermodynamics and is universally valid in practical applications.

Further major challenges arise from the super nonlinear aspect of the migration term. When the temperature is invariant (constant), and the salinity effects are neglected, the model turns out to have a well-understood dissipative structure in its most general setting ($N$ ionic species with different valences and diffusivities). In the latter situation, the time evolution of the entropy generated 
by $E:=\sum_{i=1}^{N} \int_{\TT^3} c_i \log \frac{c_i}{\bar{c}_i} dx$ (where $\bar{c}_i$ denotes the spatial mean of $c_i$ over $\TT^3$) and $\frac{\varepsilon}{2} \|\na \Psi\|_{L^2}^2$ amounts to the sum of two major components: the term $-\sum_{i=1}^{N} D_i \left\|\frac1{\sqrt{c_i}}(\na c_i + \frac{ez_i c_i \na \Psi}{k_{B} T}) \right\|_{L^2}^2 \le 0$ which speeds up the loss of energy, and a forcing term $\int_{\TT^3} -u \cdot \na \rho \Psi dx$ that gets canceled via coupling with the spatial $L^2$ time evolution of the velocity. A major tool employed in this aforementioned derivation is integration by parts, which is temperature-insensitive when $T$ is spatially-independent.  In the case of a variational temperature, the aforementioned technique breaks down, and some nonlinear terms involving the temperature and its derivatives appear and beat the dissipative structure of the system. 

In this paper, we address the case of $N$ ionic species with different valences but equal diffusivities, i.e., $D_1 = \dots = D_N = D$. This particular setting allows us to study the time evolution of the charge density $\rho$ via multiplying the concentration equations by $N_Aez_i$ and summing over all indices $i \in \left\{1, \dots, N \right\}$, in which case the diffusion terms  $-\sum_{i=1}^{N}DN_Aez_i\Delta c_i$ boil down to $-D\Delta \rho$, resulting in the following energy evolution
\be \la{introe}
\frac{\varepsilon}{2} \frac{d}{dt} \|\na \Psi\|_{L^2}^2 + \fr{D}{\varepsilon}\|\rho\|_{L^2}^2 + \sum\limits_{i=1}^{N} \fr{DN_Az_i^2e^2}{k_{B}} \int_{\TT^3} \frac{1}{T} c_i |\na \Psi|^2 dx = -\int_{\TT^3} u\cdot\nabla \rho \Psi := \mathcal{V}.
\ee 
In parallel, we turn our attention to the evolutions of $\|u\|_{L^2}^2$ and $\sum_{i=1}^{N} \int_{\TT^3} c_i \log \fr{c_i}{\bar{c}_i} dx$, and we observe that the former energy produces some bad terms that could be nontrivially controlled by the latter's dissipation and vice versa. While some expectedly unpleasant behavior arises from the electromigration term, some unexpected mess occurs due to the salinity effects. More precisely, the velocity contributes to some crucial cancellations in the coupled system (namely, it cancels $\mathcal{V}$ in \eqref{introe}), provided that the integrals $\int_{\TT^3} uc_i dx$ are in good shape. However, as the ionic concentrations do not enjoy high regularity but are only invariant in $L^1$, we seek a decomposition of the product $uc_i$ into $u \sqrt{c_i} \sqrt{c_i}$, estimate using H\"older's inequality and interpolation between $L^2$ and $H^1$. While the $L^2$ norm of $\sqrt{c_i}$ remains constant in time, that of $\na \sqrt{c_i}$ requires some challenging work to be fully dominated by the dissipation governing the evolution of the entropy $E$ and \eqref{introe}. We successfully tackle all these technicalities and construct weak solutions that are global in time for large initial data. We point out that the ideas involved in this work are new and distinguish the analysis of our system from that of the NPNS system.

The existence of weak solutions is based on a regularization scheme of the NPB model by which the nonlinearities and initial data are mollified. The smoothed-out systems have global $C^{\infty}$ solutions that are bounded in specific Lebesgue spaces, uniformly in $\eta$. However, this uniform boundedness depends upon a uniform-in-$\eta$ control of the initial energies whose evolutions are addressed. In order to achieve such an approximation, we prove new $L \log L$ estimates (Proposition \ref{fenchel}) for standard mollification operators to obtain good control of the initial mollified entropy. The proof is based on the concavity of the logarithmic function and Fenchel's inequality. In order to address the exponential decay in time of weak solutions $(c_i, u, T)$ to their steady states $(\bar{c}_i, 0, T_r)$, we make use of logarithmic Sobolev inequalities and the Csiszar-Kullback-Pinsker inequality. In addition, we state and prove a general result on the convergence of entropies (Proposition \ref{entropyconvergence}) and apply it to establish the exponential decay in time of the entropy $\int_{\TT^3} c_i \log\frac{c_i}{\bar{c}_i} dx$ to $0$. Some size constraint on $\overline{c_i(0)}$ is needed due to the salinity term through the buoyancy force.

In the general setting of different diffusivities, the nice structure $-\sum_{i=1}^N DN_Aez_i \Delta c_i = -D\Delta \rho$ breaks down, and the appearance of the crossing terms $\int_{\mathbb T^3} c_i c_j dx$ is out of control. The usual entropy machinery that is extensively used in the literature of NPNS systems does not apply to the NPB model due to the spatially-dependent temperature. Such a more general setting for the NPB system needs further investigation and falls beyond the scope of this paper.

\subsection{Organization of the paper} In Section \ref{sec:pre-results}, we introduce some notations that will be used throughout this paper, and we state the main results. In Section \ref{sec:reg}, we construct an $\eta$-approximate regularization scheme for the NPB system based on mollifications, address its global well-posedness, and establish estimates that are uniform in $\eta$. In Section \ref{sec:existence}, we prove the main result on the global existence of weak solutions to the NPB system. In Section \ref{sec:longtime}, we study the long-time dynamics of the weak solutions obtained in Section \ref{sec:existence} and the corresponding entropies, and we establish their exponential decay in time to steady states.
Finally, we summarize the 
main results of this paper and present some interesting related open problems in Section \ref{sec:conclusion}.

\section{Main Results}\label{sec:pre-results}

\subsection{Notations}
Throughout the paper, $C$ denotes a positive universal constant, and it may change from step to step. When necessary, we denote by $C_{a,b,\dots}$ a positive constant depending on $a,b,\dots$. For $p\in [1,\infty]$ we denote by $L^p(\mathbb T^3)$ the Lebesgue spaces of measurable functions $f$ from $\mathbb T^3$ to $\R$ (or $\RR^3)$ such that
\be 
\|f\|_{L^p} = \left(\int_{\mathbb T^3} |f(x)|^p dx\right)^{1/p} <\infty \,\, \text{if} \,\, p \in [1, \infty), \quad \|f\|_{L^{\infty}} = \esssup_{x\in\mathbb T^3}  |f(x)| < \infty \,\, \text{if} \,\, p = \infty.
\ee
The $L^2$ inner product is denoted by $\langle \cdot,\cdot\rangle$. For $s \in \NN$, we denote by $H^s(\mathbb T^3)$ the Sobolev spaces of measurable functions $f$ from $\mathbb T^3$ to $\R$ (or $\RR^3)$ with weak derivatives of order $s$ such that  
$
\|f\|_{H^s}^2 = \sum_{|\alpha| \le s} \|D^{\alpha}f\|_{L^2}^2 < \infty.
$  
We often write $L^p$ and $H^s$ instead of $L^p(\TT^3)$ and $H^s(\TT^3)$ for simplicity. For periodic function $f$ we have its Fourier series representation
\begin{equation}
    f = \sum\limits_{k\in \mathbb Z^3} \hat{f}_k e^{-2\pi ik\cdot x}, \quad \text{with} \quad \hat{f}_k = \int_{\TT^3} f(x) e^{2\pi ik\cdot x}.
\end{equation}
The $H^s$ norm of periodic function $f$ has an equivalent form
$
    \|f\|_{H^s}^2 = \sum_{k\in \mathbb Z^3} (1+|k|^{2s})|\hat{f}_k|^2.
$
We denote the dual space of $H^s$ by $H^{-s}$ for $s\in \mathbb N$.

For a Banach space $(X, \|\cdot\|_{X})$ and $p\in [1,\infty]$, we consider the Lebesgue spaces $ L^p(0,T; X)$ of functions $f$  from $X$ to $\R$ (or $\RR^3)$ satisfying 
$
\int_{0}^{T} \|f\|_{X}^p dt  <\infty
$ with the usual convention when $p = \infty$. 

Denote by
$
\mathcal V := \left\{u \in C^\infty(\mathbb T^3): \na \cdot u = 0, \, \int_{\TT^3} u dx = 0 \right\},
$ 
and let $H$ and $V$ be the $L^2$ and $H^1$ closure of $\mathcal V$. Let $P_\sigma$ be the Leray projection and $A=-P_\sigma \Delta$ be the Stokes operator. Then the domain $\mathcal D(A) = H^2\cap V$. We denote dual spaces of $V$ and $\mathcal D(A)$ by $V'$ and $\mathcal D(A)'$, respectively.

By combining $\nabla p$ with $g\vec{k}=\nabla (gz)$ as a new pressure gradient and taking $D_1=D_2=\dots=D_N=D$, the NPB system reads
\noeqref{eqn:npb-full-1}
\noeqref{eqn:npb-full-2}
\noeqref{eqn:npb-full-3}
\noeqref{eqn:npb-full-4}
\noeqref{eqn:npb-full-5}
\begin{subequations}\label{sys:npb-full}
    \begin{align}
        &\partial_t c_i + u\cdot \nabla c_i  - D\Delta c_i - D\frac{e}{k_B}\nabla\cdot (z_i\frac{1}{T} c_i \nabla \Psi)= 0,\label{eqn:npb-full-1}
        \\
        &-\varepsilon\Delta \Psi = \rho = \sum\limits_{i=1}^N N_Ae z_i c_i, \label{eqn:npb-full-2}
        \\
        &\partial_t u + u\cdot \nabla u - \nu\Delta u + \nabla p = g\left(\alpha_T(T-T_r) - \alpha_S\left(\sum\limits_{i=1}^N c_i M_i-S_r\right)\right) \vec{k} - \rho \nabla\Psi, \label{eqn:npb-full-3}
        \\
        &\nabla\cdot u = 0,\label{eqn:npb-full-4} 
        \\
        &\partial_t T + u\cdot \nabla T - \kappa\Delta T = 0.\label{eqn:npb-full-5}
    \end{align}
\end{subequations}
We denote the initial conditions of $(c_i,u,T)$ by $(c_i(0),u_0,T_0)$.
In the periodic case, the reference temperature $T_r$ and reference salinity $S_r$ are 
$T_r= \bar{T}$ and $ S_r = \sum_{i=1}^N \bar{c}_i M_i$,
where $\bar{T} := \frac1{\int_{\mathbb T^3} 1 dx}\int_{\mathbb T^3} T dx = \int_{\mathbb T^3} T dx$ and $\bar{c}_i := \frac1{\int_{\mathbb T^3} 1 dx}\int_{\mathbb T^3} c_i dx= \int_{\mathbb T^3} c_i dx$ are invariant in time. These in turn give that $\bar{u}=\frac1{\int_{\mathbb T^3} 1 dx}\int_{\mathbb T^3} u dx=\int_{\mathbb T^3} u dx$ is invariant in time. Therefore, $\bar{u}_0=0$ implies that $\bar{u}=0$ for any time. 

From \eqref{eqn:npb-full-2}, we have the following relation between $\Psi_0$ and $\rho_0$ in terms of $c_i(0)$:
\begin{equation}
    -\varepsilon \Delta \Psi_0 = \rho_0 = \sum_{i=1}^N N_Aez_i c_i(0).
\end{equation}
Due to the periodic boundary conditions, we see that $\int_{\TT^3} \rho dx =0$ for all times from \eqref{eqn:npb-full-2}. Therefore,
we need to impose a compatibility condition on the initial concentrations:
\begin{equation}\label{eqn:compa}
    \sum\limits_{i=1}^N N_Aez_i \overline{c_i(0)} = 0.
\end{equation}

\subsection{Main results} The following theorem concerns the global existence of weak solutions to  \eqref{sys:npb-full} in $\TT^3$.
\begin{thm}\label{thm:main-torus}
   Suppose that $c_i(0)$ is nonnegative and satisfies the compatibility condition \eqref{eqn:compa}, and $T_0$ is bounded below by $T_0\geq T^* >0$. Assume that $\nabla\Psi_0\in L^2$,
    $c_i(0)\in L^1$, $c_i(0)\log c_i(0)\in L^1$, $u_0\in H$, and $T_0\in L^{3+\delta}$ for a fixed and arbitrarily small $\delta>0$. Then for any time $\mathcal T>0$ there exists a weak solution $(c_i,u,T)$ to system \eqref{sys:npb-full} on $[0,\mathcal T]$ satisfying
    \begin{equation}\label{regularity-thm}
        \begin{split}
            &u\in L^\infty(0,\mathcal T; H) \cap L^2(0,\mathcal T; V), \quad 
         T\in L^\infty(0,\mathcal T; L^{3+\delta}) \cap L^2(0,\mathcal T; H^1), 
         \\
        &c_i\in L^\infty(0,\mathcal T; L^1) \cap L^2(0,\mathcal T; L^{\frac32})\cap L^\frac43(0,\mathcal T;  W^{1,\frac65}),
        \end{split}
    \end{equation}
    with the property that $c_i(x,t)\geq 0$ and $T(x,t)\geq T^*$ for a.e. $(x,t)\in\mathbb T^3\times [0,\mathcal T]$, and such that
    \noeqref{thm:eqn2}
    \begin{align}
        &\int_0^{\mathcal T} \left(-\langle c_i, \partial_t \phi\rangle - \langle u \cdot \nabla \phi, c_i\rangle + D\langle \nabla c_i + \frac{e}{k_B}z_i\frac1Tc_i\nabla\Psi, \nabla \phi\rangle \right) dt  = \langle c_i(0),\phi(0)\rangle, \label{thm:eqn1}
        \\
        &\int_0^{\mathcal T} \left( -\langle u,\partial_t \varphi\rangle - \langle u\cdot\nabla\varphi, u\rangle + \nu \langle \nabla u, \nabla\varphi\rangle \right) dt 
        \\
        &\hspace{1cm}= \langle u_0, \varphi(0)\rangle + \int_0^{\mathcal T}  \Big( \langle g(\alpha_T(T-T_r) - \alpha_S(\sum\limits_{i=1}^N c_i M_i-S_r)) ,\varphi_3 \rangle -  \langle \rho\nabla\Psi, \varphi\rangle \Big) dt, \label{thm:eqn2}
        \\
        &\int_0^{\mathcal T} \left(-\langle T, \partial_t \phi\rangle - \langle u \cdot \nabla \phi, T\rangle + \kappa\langle \nabla T, \nabla \phi\rangle \right) dt  = \langle T_0,\phi(0)\rangle, \label{thm:eqn3}
    \end{align}
    for any scalar test function $\phi$ and vector-valued test function $\varphi$ such that $\phi, \varphi \in C^\infty([0,\mathcal T]\times \TT^3)$ with $\phi(T)=\varphi(T)=0$ and $\nabla\cdot \varphi=0$, where $\Psi$ and $\rho$ satisfy \eqref{eqn:npb-full-2}.
\end{thm}

\begin{rem}
    The $L^{3+\delta}$ requirement for $T_0$ is due to the cubic nonlinearity in the electromigration term.
\end{rem}

The next theorem is about the long-time dynamics for weak solutions to system \eqref{sys:npb-full}.
\begin{thm}\label{thm:longtime}
    Suppose that $(c_i, u, T)$ is a weak solution to system \eqref{sys:npb-full} obtained in Theorem \ref{thm:main-torus}. Then the temperature $T(t)$ decays exponentially in time to the initial average $\overline{T_0}=T_r$ in $L^2$.
    Moreover, if the initial averages of the concentrations $\overline{c_i(0)}$ satisfy the size condition
    \begin{equation}
        \max\limits_{i\in\{1,\dots,N\}} \overline{c_i(0)} \leq \frac{Dk_B N_A T^* \nu}{4C\alpha_S^2} = \frac{DR T^* \nu}{4C\alpha_S^2},
    \end{equation}
    where $C$ is some constant depending only on the domain $\TT^3$, then $c_i(t)$ and $u(t)$ decay exponentially in time to $\overline{c_i(0)}$ and $0$ in $L^1$ and $L^2$, respectively, and the entropy $\int_{\TT^3} c_i \log\frac{c_i(t)}{\bar{c}_i} dx$ decays exponentially in time to $0$. In particular, if the salinity effect in the buoyancy is negligible, i.e., $\alpha_S=0$, then the decay holds unconditionally without any size assumptions on $\overline{c_i(0)}$.
\end{thm}

\section{Regularization Scheme}\label{sec:reg}

\subsection{Mollifiers} 
We consider a family of standard periodic mollifiers $\left\{\phi_{\eta}  \right\}_{\eta \in (0,1)}$ with $\int_{\TT^3} \phi_{\eta}(x) dx = 1$ for all $\eta \in  (0,1)$, and we denote by $\mathcal J_{\eta}$ the convolution mollification operator $\mathcal{J}_{\eta} f = \phi_{\eta} * f = \int_{\TT^3} \phi_\eta(x-y) f(y) dy$. To be more specific, $\{\phi_\eta\}_{\eta \in (0,1)}$ are periodic smooth functions with Fourier series given by  $\phi_\eta = \sum_{k\in \mathbb Z^3}  e^{-\eta |k|^2} e^{-2\pi ik\cdot x}$. It follows that 
$
\widehat{(\mathcal J_\eta f)}_k = \widehat{(\phi_\eta * f)}_k = e^{-\eta |k|^2} \hat{f}_k.
$
Moreover, $\mathcal{J}_{\eta}$ is bounded on all $L^p$ spaces for any $p \in [1,\infty]$ and uniformly in $\eta$, and it holds that 
\begin{align}\label{eqn:comm-of-convolution}
    \langle \mathcal J_\eta f,g \rangle = \int_{\mathbb T^3} \mathcal J_\eta  f(x) g(x) dx 
    = \int_{\mathbb T^3}  f(x) \mathcal J_\eta  g(x) dx = \langle f, \mathcal J_\eta g \rangle ,
\end{align} for any $\eta \in (0,1)$ due to Fubini's theorem. This implies that $\int_{\TT^3} \mathcal J_\eta f dx = \int_{\TT^3} f dx$ for any $\eta \in (0,1)$.

We provide below $L \log L$ estimates, uniform in $\eta$, that will be used in the subsequent subsections:

\beg{prop} \la{fenchel} Let $f \in L^1(\TT^3)$ be a nonnegative function with $\int_{\TT^3} f \log f dx < \infty$. For any $\eta \in (0,1)$, it holds that
\be 
\int_{\TT^3} \mathcal J_{\eta} f \log \mathcal J_{\eta}f dx \le \int_{\TT^3} f \log f dx. 
\ee  
\end{prop}

\begin{proof}
   In view of the self-adjointness of the operators $\mathcal J_{\eta}$, we have 
   \be 
\int_{\TT^3} \mathcal J_{\eta} f \log \mathcal J_{\eta}f dx
= \int_{\TT^3} f \mathcal J_{\eta} \log \mathcal J_{\eta}f dx. 
   \ee An application of Jensen's inequality to the probability measure $\phi_{\eta}dy$ and the concave logarithmic function yields the pointwise estimate 
   \be 
\mathcal J_{\eta} \log \mathcal J_{\eta} f(x) 
= \int_{\TT^3} \phi_{\eta} (x-y) \log \mathcal J_{\eta} f(y) dy 
\le \log \left(\int_{\TT^3} \phi_{\eta} (x-y) \mathcal J_{\eta} f(y) dy\right)
= \log \mathcal J_{\eta} \mathcal{J}_{\eta} f (x),
   \ee from which we infer that 
   \be 
\int_{\TT^3} \mathcal J_{\eta} f \log \mathcal J_{\eta}f dx = \int_{\TT^3} f\mathcal J_{\eta} \log \mathcal J_{\eta}f dx
\le \int_{\TT^3} f\log \mathcal J_{\eta} \mathcal{J}_{\eta} f dx.
   \ee 
Since the Legendre transform of $p \log p - p + 1$ is $e^q - 1$ for $p \ge 0$, it follows from Fenchel's inequality that $pq \le p \log p - p  + e^q$ for any $p \ge 0$. By taking $p = f$ and $q = \log \mathcal J_{\eta} \mathcal J_{\eta} f$, we obtain the bound
\be 
\int_{\TT^3} \mathcal J_{\eta} f \log \mathcal J_{\eta}f dx
\le \int_{\TT^3} \left(f \log f - f + \mathcal{J}_{\eta} \mathcal{J}_{\eta} f \right) dx
= \int_{\TT^3} f \log f dx.
\ee 
\end{proof}

\subsection{The mollified NPB system.}
In order to study system \eqref{sys:npb-full}, we consider the following mollified version:
\noeqref{eqn:npb-full-mo-1}
\noeqref{eqn:npb-full-mo-2}
\noeqref{eqn:npb-full-mo-3}
\noeqref{eqn:npb-full-mo-4}
\noeqref{eqn:npb-full-mo-5}
\begin{subequations}\label{sys:npb-full-mo}
    \begin{align}
        &\partial_t c_i^\eta + \mathcal J_\eta  u^\eta\cdot \nabla c_i^\eta  - D\Delta c_i^\eta - D\frac{e}{k_B}\nabla\cdot (z_i\frac{1}{T^\eta} c_i^\eta \nabla \mathcal J_\eta \mathcal J_\eta \Psi^\eta)= 0,\label{eqn:npb-full-mo-1}
        \\
        &-\varepsilon\Delta \Psi^\eta = \rho^\eta = \sum\limits_{i=1}^N e N_A z_i c_i^\eta, \label{eqn:npb-full-mo-2}
        \\
        &\partial_t u^\eta + \mathcal J_\eta u^\eta\cdot \nabla u^\eta - \nu\Delta u^\eta + \nabla p^\eta 
        \\
         &\hspace{2cm}= g\left(\alpha_T(T^\eta-T^\eta_r) - \alpha_S\left(\sum\limits_{i=1}^N c^\eta_i M_i-S^\eta_r\right)\right) \vec{k} -\mathcal J_\eta  \left( \rho^\eta \nabla \mathcal J_\eta   \mathcal J_\eta   \Psi^\eta\right), \label{eqn:npb-full-mo-3}
        \\
        &\nabla\cdot u^\eta = 0,\label{eqn:npb-full-mo-4} 
        \\
        &\partial_t T^\eta +\mathcal J_\eta   u^\eta\cdot \nabla T^\eta - \kappa\Delta T^\eta = 0,\label{eqn:npb-full-mo-5}
    \end{align}
\end{subequations}
with initial conditions $(c_i^\eta(0),u^\eta_0,T^\eta_0)=(\mathcal J_\eta c_i(0),\mathcal J_\eta u_0,\mathcal J_\eta T_0)$.
Note that they are all smooth (in $C^\infty(\mathbb T^3)$) provided that $c_i(0), u_0, T_0$ are $L^1$ integrable. Here $T_r^\eta = \overline{T^\eta}$ and $S_r^\eta = \sum\limits_{i=1}^N \overline{c_i^\eta} M_i$.
According to the property of the mollifier, the conditions $c_i(0)\geq 0$, $T_0\geq T^*>0$, and $\overline{u_0}=0$ imply that $c_i^\eta(0)\geq 0$, $T^\eta_0\geq T^*>0$, and $\overline{u^\eta_0}=0$. Also, $\overline{T^\eta} = \overline{T^\eta_0} = \overline{T_0}$ and $\overline{c_i^\eta} = \overline{c_i^\eta(0)}= \overline{c_i(0)}$ are invariant in time. Based on these  observations, together with property \eqref{eqn:comm-of-convolution}, we have $\overline{u^\eta}=\overline{u^\eta_0}=0$, and thus $u^\eta \in H$. The goal of this section is to study the global well-posedness of system \eqref{sys:npb-full-mo} and derive the uniform-in-$\eta$ estimates for this system.

We start by addressing the global well-posedness of solutions to the mollified system \eqref{sys:npb-full-mo}.

\begin{prop}\label{thm:mollified-sys}
Suppose that $c_i(0)$ is nonnegative and satisfying the compatibility condition \eqref{eqn:compa}, and $T_0$ is bounded below by $T_0\geq T^* >0$. Assume that $\nabla\Psi_0\in L^2$,
    $c_i(0)\in L^1$, $c_i(0)\log c_i(0)\in L^1$, $u_0\in H$, and $T_0\in L^2$.
    For each $\eta>0$, consider system \eqref{sys:npb-full-mo} with initial conditions $(c_i^\eta(0),u^\eta_0,T^\eta_0)=(\mathcal J_\eta c_i(0),\mathcal J_\eta u_0,\mathcal J_\eta T_0)$. Then for any $\mathcal T>0$, there exists a unique strong solution $(c_i^\eta,u^\eta,T^\eta)$ to system \eqref{sys:npb-full-mo} on $[0,\mathcal T]$ such that
    \begin{align}
        c_i^\eta, T^\eta \in C([0,\mathcal T];C^\infty(\mathbb T^3)), \quad \text{and} \quad u^\eta \in C([0,\mathcal T];C^\infty(\mathbb T^3)\cap H).
    \end{align}
Moreover, $c_i^{\eta}(t) \ge 0$ and $T^{\eta}(t) \ge T^*$ for any $t \in [0,\mathcal{T}]$.
\end{prop}

The proof of Proposition \ref{thm:mollified-sys} will be presented in Appendix \ref{sec:appendix-a}. 

The following proposition addresses the uniform-in-$\eta$ boundedness of solutions.

\begin{prop}\label{prop:uni-est-sol} 
    Suppose that $c_i(0)$ is nonnegative and satisfying the compatibility condition \eqref{eqn:compa}, and $T_0$ is bounded below by $T_0\geq T^* >0$. Assume that $\nabla\Psi_0\in L^2$,
    $c_i(0)\in L^1$, $c_i(0)\log c_i(0)\in L^1$, $u_0\in H$, and $T_0\in L^p$ for some $p\geq 2$. Then the sequence of solutions $(u^\eta, c_i^\eta, T^\eta)$ and the corresponding $\Psi^\eta$ and $\rho^\eta$ satisfy the following uniform-in-$\eta$ bounds:
    \begin{align}
        &\nabla \mathcal J_\eta \Psi^\eta\,\,\text{and} \,\,\sqrt{c_i^\eta}\,\,  \text{are uniformly bounded in} \,\,  L^\infty(0,\mathcal T; L^2)\cap L^2(0,\mathcal T; H^1),
        \\
        &u^\eta \,\,  \text{are uniformly bounded in} \,\,  L^\infty(0,\mathcal T; H)\cap L^2(0,\mathcal T; V),
        \\
        & T^\eta \,\,  \text{are uniformly bounded in} \,\, L^\infty(0,\mathcal T; L^p)\cap L^2(0,\mathcal T; H^1),
        \\
        &\Psi^\eta \,\,  \text{are uniformly bounded in} \,\,  L^2(0,\mathcal T; W^{2,\frac32}) \cap L^\frac43(0,\mathcal T; W^{3,\frac65}), 
        \\
        &c_i^\eta \,\,\text{and} \,\, \rho^\eta  \,\,  \text{are uniformly bounded in} \,\,  L^2(0,\mathcal T; L^{\frac32}) \cap L^\frac43(0,\mathcal T;  W^{1,\frac65}).
    \end{align}
\end{prop}
\begin{proof}
We drop the superscript $\eta$ to simplify notations (unless needed). The proof will be performed in several steps.

\smallskip
\noindent{\bf Step 1. Bounds of $T$.} The $L^2$ and $L^p$ norms of $T$ satisfy 
\begin{equation}
    \|T(t)\|_{L^2}^2 + 2\kappa \int_0^t \|\nabla T(s)\|_{L^2}^2 ds = \|T_0\|_{L^2}^2, \quad \|T(t)\|_{L^p} \leq \|T_0\|_{L^p},
\end{equation}
which imply that
$T^\eta$ are uniformly bounded in $ L^\infty(0,\mathcal T; L^p)\cap L^2(0,\mathcal T; H^1).$
Moreover, as $T_r=\bar{T}$ is invariant in time, we have the following Poincar\'e inequality
\[
\|T-T_r\|_{L^2}^2 \leq C_p \|\nabla T\|_{L^2}^2.
\]
Taking the $L^2$ inner product of the equation obeyed by $T$ with $T-T_r$ and using the Poincar\'e inequality above, we obtain
\[
\frac{1}{2} \frac{d}{d t}\|T-T_r\|_{L^2}^2+ \frac{\kappa}{C_p}\|T-T_r\|_{L^2}^2 
\le0,
\]
yielding the uniform bound 
\begin{equation}\label{est:T}
    \|T(t)-T_r\|_{L^2}^2 \leq \|T_0-T_r\|_{L^2}^2 e^{-\frac{2\kappa}{C_p} t}.
\end{equation}

\noindent{\bf Step 2. Bounds of $\|\nabla \mathcal J_\eta \Psi\|_{L^2}$ and $\|u\|_{L^2}$.}
The ionic concentrations evolve according to the equations
$$
\partial_t\left(eN_Az_i c_i\right)+\mathcal J_\eta  u \cdot \nabla\left(eN_Az_i c_i\right)-D \Delta\left(eN_Az_i c_i\right)=D\frac{e^2}{k_B} \nabla \cdot\left(N_Az_i^2 c_i \frac1T\nabla \mathcal J_\eta \mathcal J_\eta \Psi\right),
$$
for all $i \in\{1, \ldots, N\}$. Summing over all indices $i \in \left\{1, \dots, N \right\}$, we deduce that the charge density $\rho$ evolves according to
$$
\partial_t \rho+ \mathcal J_\eta  u \cdot \nabla \rho-D \Delta \rho=D\frac{e^2}{k_B} \sum_{i=1}^N \nabla \cdot\left(N_Az_i^2 c_i \frac1T\nabla \mathcal J_\eta \mathcal J_\eta \Psi\right).
$$ Here the diffusion term $-\Delta \rho$ is obtained based on the fact that all ionic species have equal diffusivities.  
Taking the $L^2$ inner product of this latter equation with $\mathcal J_\eta  \mathcal J_\eta  \Psi$ and  integrating by parts, we have
$$
\frac{\varepsilon}{2} \frac{d}{d t}\left\|\nabla\mathcal J_\eta \Psi\right\|_{L^2}^2+ \frac{D}{\varepsilon}\|\mathcal J_\eta \rho\|_{L^2}^2=-\int_{\mathbb{T}^3}(\mathcal J_\eta  u \cdot \nabla \rho) \mathcal J_\eta \mathcal J_\eta \Psi dx -D\frac{e^2N_A}{k_B} \int_{\mathbb{T}^3} \sum_{i=1}^N z_i^2 c_i \frac1T |\nabla \mathcal J_\eta \mathcal J_\eta \Psi|^2 dx.
$$ 

The $L^2$ norm of the velocity $u$ satisfies
$$
\frac{1}{2} \frac{d}{d t}\|u\|_{L^2}^2+\nu\|\nabla u\|_{L^2}^2=-\int_{\mathbb{T}^3} \rho \nabla\mathcal J_\eta \mathcal J_\eta \Psi \cdot \mathcal J_\eta   u + \int_{\mathbb{T}^3} \left(g\alpha_T (T-T_r) u_3 -g\alpha_S(\sum\limits_{i=1}^N c_i M_i - S_r) u_3  \right)dx,
$$
where we have used \eqref{eqn:comm-of-convolution}.
Adding these evolution equations, integrating by parts, and using the divergence-free condition obeyed by the velocity vector field, we obtain the cancellation law
$$
\int_{\mathbb{T}^3}(\mathcal J_\eta  u \cdot \nabla \rho) \mathcal J_\eta \mathcal J_\eta \Psi+\int_{\mathbb{T}^3} \rho \nabla \mathcal J_\eta \mathcal J_\eta \Psi \cdot \mathcal J_\eta  u=0 .
$$
The nonnegativity of the ionic concentrations and the positivity of the temperature imply that
\begin{equation}\label{dissipation-1}
    DN_A \frac{e^2}{k_B}\sum_{i=1}^N\int_{\mathbb{T}^3}  z_i^2 c_i T^{-1}|\nabla \mathcal J_\eta \mathcal J_\eta \Psi|^2 = DN_A \frac{e^2}{k_B}\sum_{i=1}^N \|z_i \sqrt{c_i} T^{-\frac12} {\na \mathcal J_\eta \mathcal J_\eta \Psi }\|_{L^2}^2 \geq 0.
\end{equation}
By the Cauchy-Schwarz inequality, the Poincar\'e inequality, and the fact that $\bar{u}=0$, we obtain
\[
\int_{\mathbb{T}^3} g\alpha_T (T-T_r) u_3  dx \leq 
\frac\nu4\|\nabla u\|_{L^2}^2 + C\|T-T_r\|_{L^2}^2.
\]
For the salinity term $S$, by applying the Poincar\'e inequality, the interpolation inequality, and the Sobolev embedding $W^{1,1}\hookrightarrow L^{\frac32}$, we obtain
\begin{equation}\label{est:salinity}
    \begin{split}
        \left|\int_{\mathbb T^3} g\alpha_S(\sum\limits_{i=1}^N c_i M_i - S_r) u_3 dx\right| 
    \leq &C \sum\limits_{i=1}^N \|c_i-\bar{c}_i\|_{L^\frac32} \|u\|_{L^3} \leq C\sum\limits_{i=1}^N \|c_i-\bar{c}_i\|_{W^{1,1}} \|u\|_{L^2}^{\frac12}\|\nabla u\|_{L^2}^{\frac12} 
    \\
    \leq & C\sum\limits_{i=1}^N \|\nabla c_i\|_{L^1} \|u\|_{L^2}^{\frac12}\|\nabla u\|_{L^2}^{\frac12}  \leq C \|\sqrt{c_i}\|_{L^2} \|\nabla \sqrt{c_i}\|_{L^2} \|u\|_{L^2}^{\frac12}\|\nabla u\|_{L^2}^{\frac12} 
    \\
    \leq &C\sum\limits_{i=1}^N \|c_i\|_{L^1}^{\frac12} \|\nabla \sqrt{c_i}\|_{L^2} \|u\|_{L^2}^{\frac12}\|\nabla u\|_{L^2}^{\frac12}  
    \\
    \leq &\frac\nu4\|\nabla u\|^2_{L^2} + C\sum\limits_{i=1}^N\|c_i\|^{\frac23}_{L^1} \|\nabla \sqrt{c_i}\|^{\frac43}_{L^2} \| u\|^{\frac23}_{L^2}, 
    \end{split}
\end{equation}
 where we have used $\nabla c_i= 2(\nabla\sqrt{c_i}) \sqrt{c_i}$. Combining the estimates above, we obtain that
\begin{equation}\label{est:1}
    \begin{split}
        &\frac{d}{dt}({\varepsilon}\left\|\nabla\mathcal J_\eta \Psi\right\|_{L^2}^2 +  \|u\|_{L^2}^2) + \frac{2D}{\varepsilon} \|\mathcal J_\eta \rho\|_{L^2}^2 + \nu \|\na u\|_{L^2}^2 + \frac{2DN_Ae^2}{k_B}\sum_{i=1}^N \|z_i \sqrt{c_i} T^{-\frac12} {\na \mathcal J_\eta \mathcal J_\eta \Psi }\|_{L^2}^2 
    \\
    \leq & C\|T-T_r\|_{L^2}^2 + C\sum\limits_{i=1}^N\|c_i\|^{\frac23}_{L^1} \|\nabla \sqrt{c_i}\|^{\frac43}_{L^2} \| u\|^{\frac23}_{L^2}.
    \end{split}
\end{equation}

In order to control the second term on the right-hand side of \eqref{est:1}, we consider the energies 
\be 
E_i(t)= \int_{\TT^3} c_i(t) \log \fr{c_i(t)}{\bar{c}_i} dx =\int_{\TT^3} \left(c_i(t) \log \fr{c_i(t)}{\bar{c}_i} - c_i(t) + \bar{c}_i\right) dx, \quad E(t)=\sum\limits_{i=1}^N E_i(t),
\ee 
and study their evolutions.

\smallskip
\noindent{\bf Step 3. Bounds of $E$.}
In view of the algebraic inequality $x\log x - x+1\geq 0$ that holds for any $x\geq 0$, one has that $\bar{c}_i\left(\frac{c_i}{\bar{c}_i} \log \fr{c_i(t)}{\bar{c}_i} - \frac{c_i}{\bar{c}_i} + 1 \right)\geq 0$ and thus $E\geq 0$.
Differentiating $E_i(t)$ in time, we have 
\be 
\int_{\TT^3} (\pa_t c_i) \log \left(\fr{c_i}{\bar{c}_i}\right) dx
= \int_{\TT^3} \pa_t \left(c_i \log \fr{c_i}{\bar{c}_i} - c_i + \bar{c}_i\right) dx
= \fr{d}{dt} \int_{\TT^2} \left(c_i \log \fr{c_i}{\bar{c}_i} - c_i + \bar{c}_i\right) dx,
\ee which, together with the nonlinearity cancellation law
\be 
\int_{\TT^3} (\mathcal J_\eta  u \cdot \na c_i) \log \frac{c_i}{\bar{c}_i} dx
= - \int_{\TT^3} (\mathcal J_\eta  u \cdot \na 
\log c_i) c_i dx
= - \int_{\TT^3} \mathcal J_\eta  u \cdot \na c_i dx
= 0,
\ee gives rise to the energy equality
\begin{align*}
    \fr{d}{dt} E_i 
&= D \int_{\TT^3} \na \cdot \left(\na c_i + \frac{e}{k_B}z_ic_i T^{-1} \na \mathcal J_\eta \mathcal J_\eta \Psi \right) \log \frac{c_i}{\bar{c}_i} dx
\\
&= - D \int_{\TT^3}  \left(\na c_i + \frac{e}{k_B}z_ic_i T^{-1} \na \mathcal J_\eta \mathcal J_\eta \Psi \right) \cdot \fr{\na c_i}{c_i} dx 
\\
&= - D \int_{\TT^3} \fr{1}{c_i} \left(\na c_i + \frac{e}{k_B}z_ic_i T^{-1} \na \mathcal J_\eta \mathcal J_\eta \Psi \right) \cdot \left(\na c_i + \frac{e}{k_B}z_ic_i T^{-1} \na \mathcal J_\eta \mathcal J_\eta \Psi \right) dx
\\
&\hspace{0.3in}+ D \int_{\TT^3} \fr{1}{c_i}\left(\na c_i + \frac{e}{k_B}z_ic_i T^{-1} \na \mathcal J_\eta \mathcal J_\eta \Psi \right) \cdot \left( \frac{e}{k_B}z_ic_i T^{-1} \na \mathcal J_\eta \mathcal J_\eta \Psi \right) dx
\\&= - D \left\|\fr{\na c_i +\frac{e}{k_B}z_ic_i T^{-1} \na \mathcal J_\eta \mathcal J_\eta \Psi}{\sqrt{c_i}} \right\|_{L^2}^2
\\
&\hspace{0.3in}+ D\int_{\TT^3} \left(\na c_i + \frac{e}{k_B}z_ic_i T^{-1} \na \mathcal J_\eta \mathcal J_\eta \Psi \right) \cdot \left( \frac{e}{k_B} z_iT^{-1} \na \mathcal J_\eta \mathcal J_\eta \Psi \right) dx.
\end{align*} 
Summing over all indices $i \in \left\{1, \dots, N \right\},$ we infer that 
\begin{align*}
    &\frac d{dt} E + D\sum\limits_{i=1}^N \left\|\fr{\na c_i +\frac{e}{k_B}z_ic_i T^{-1} \na \mathcal J_\eta \mathcal J_\eta \Psi}{\sqrt{c_i}} \right\|_{L^2}^2
    \\
    =& D \sum\limits_{i=1}^N\int_{\TT^3} \left(\na c_i + \frac{e}{k_B}z_ic_i T^{-1} \na \mathcal J_\eta \mathcal J_\eta \Psi \right) \cdot \left(\frac{e}{k_B} z_iT^{-1} \na \mathcal J_\eta \mathcal J_\eta \Psi \right) dx
    \\
    =& D \sum\limits_{i=1}^N\int_{\TT^3} \frac{\na c_i + \frac{e}{k_B}z_ic_i T^{-1} \na \mathcal J_\eta \mathcal J_\eta \Psi}{\sqrt{c_i}} \cdot \left( \frac{e}{k_B} z_i\sqrt{c_i}T^{-\frac12} \na \mathcal J_\eta \mathcal J_\eta \Psi \right) T^{-\frac12} dx
    \\
    \leq &\frac{D}{2} \sum\limits_{i=1}^N \left\|\fr{\na c_i +\frac{e}{k_B}z_ic_i T^{-1} \na \mathcal J_\eta \mathcal J_\eta \Psi}{\sqrt{c_i}} \right\|_{L^2}^2 + \frac{De^2}{2k_B^2} \|T^{-\frac12}\|_{L^\infty}^2 \sum_{i=1}^N  \left\|z_i \sqrt{c_i} T^{-\frac12} {\na \mathcal J_\eta \mathcal J_\eta \Psi }\right\|_{L^2}^2.
\end{align*}
Since $\|T^{-\frac12}\|_{L^\infty}\leq \frac1{\sqrt{T^*}}$, the above yields 
\begin{align}\label{est:2}
   \frac d{dt}  E + \frac{D}{2} \sum\limits_{i=1}^N \left\|\fr{\na c_i +\frac{e}{k_B}z_ic_i T^{-1} \na \mathcal J_\eta \mathcal J_\eta \Psi}{\sqrt{c_i}} \right\|_{L^2}^2 
   \leq \frac{De^2}{2k_B^2 T^*}\sum_{i=1}^N \|z_i \sqrt{c_i} T^{-\frac12} {\na \mathcal J_\eta \mathcal J_\eta \Psi }\|_{L^2}^2.
\end{align} 

\noindent
{\bf Step 4. Combining estimates.}
Observe that the dissipation in \eqref{est:1} can control the right-hand side of \eqref{est:2}. We multiply \eqref{est:2} by $2k_B N_A T^*$ and add it to the energy inequality \eqref{est:1} to obtain
\begin{equation}\label{est:5}
    \begin{split}
         &\frac{d}{dt}({\varepsilon}\left\|\nabla\mathcal J_\eta \Psi\right\|_{L^2}^2 +  \|u\|_{L^2}^2 + 2k_BN_A{T^*} E) + \frac{2D}{\varepsilon} \|\mathcal J_\eta \rho\|_{L^2}^2 + \nu \|\na u\|_{L^2}^2 
   \\
   &\hspace{0.2in}+ \frac{De^2}{k_B}N_A\sum_{i=1}^N \|z_i \sqrt{c_i} T^{-\frac12} {\na \mathcal J_\eta \mathcal J_\eta \Psi }\|_{L^2}^2 + DN_A k_B T^* \sum\limits_{i=1}^N \left\|\fr{\na c_i +\frac{e}{k_B}z_ic_i T^{-1} \na \mathcal J_\eta \mathcal J_\eta \Psi}{\sqrt{c_i}} \right\|_{L^2}^2 
    \\
    \leq & C\|T-T_r\|_{L^2}^2 + C\sum\limits_{i=1}^N\|c_i\|^{\frac23}_{L^1} \|\nabla \sqrt{c_i}\|^{\frac43}_{L^2} \| u\|^{\frac23}_{L^2}. 
    \end{split}
\end{equation}
In order to control $\|\nabla\sqrt{c_i}\|_{L^2}$, we observe that 
\begin{align*}
    \|2\nabla \sqrt{c_i}\|_{L^2}^2 &=  \|2\nabla \sqrt{c_i} + \frac{e}{k_B}z_i \sqrt{c_i} T^{-1}\na \mathcal J_\eta \mathcal J_\eta \Psi  - \frac{e}{k_B}z_i \sqrt{c_i} T^{-1}\na \mathcal J_\eta \mathcal J_\eta \Psi\|_{L^2}^2
    \\
    &\leq 2 \|2\nabla \sqrt{c_i} + \frac{e}{k_B} z_i \sqrt{c_i} T^{-1} \na \mathcal J_\eta \mathcal J_\eta \Psi\|_{L^2}^2 + 2 \|T^{-\frac12}\|_{L^\infty}^2 \|\frac{e}{k_B} z_i \sqrt{c_i} T^{-\frac12}\na \mathcal J_\eta \mathcal J_\eta \Psi\|_{L^2}^2
    \\
    &\leq 2 \left\|\frac{\nabla c_i +\frac{e}{k_B} z_i c_i T^{-1} \na \mathcal J_\eta \mathcal J_\eta \Psi}{\sqrt{c_i}}\right\|_{L^2}^2 + 2 \frac{e^2}{k_B^2{T^*}}\left\|z_i \sqrt{c_i} T^{-\frac12}\na \mathcal J_\eta \mathcal J_\eta \Psi\right\|_{L^2}^2,
\end{align*}
where we have used $\|T^{-\frac12}\|_{L^\infty}\leq \frac1{\sqrt{T^*}}$ and $2\nabla\sqrt{c_i} = \frac{\nabla c_i}{\sqrt{c_i}}$. By multiplying by $\frac{Dk_BN_A T^*}4$ on both sides above, applying Young's inequality to obtain $C\|c_i\|^{\frac23}_{L^1} \|\nabla \sqrt{c_i}\|^{\frac43}_{L^2} \| u\|^{\frac23}_{L^2} \leq \frac{DN_Ak_BT^*}2 \|\nabla \sqrt{c_i}\|^2 + C\|c_i\|^{2}_{L^1}\| u\|^{2}_{L^2}$,
we can infer from \eqref{est:5} that
\begin{equation}\label{est:6}
    \begin{split}
         &\frac{d}{dt}({\varepsilon}\left\|\nabla\mathcal J_\eta \Psi\right\|_{L^2}^2 +  \|u\|_{L^2}^2 + 2k_BN_A{T^*} E) + \frac{2D}{\varepsilon} \|\mathcal J_\eta \rho\|_{L^2}^2 + \nu \|\na u\|_{L^2}^2 + \frac{DN_Ak_BT^*}2 \sum\limits_{i=1}^N\|\nabla \sqrt{c_i}\|_{L^2}^2
   \\
   &\hspace{0.2in}+ \frac{De^2}{2k_B}N_A\sum_{i=1}^N \|z_i \sqrt{c_i} T^{-\frac12} {\na \mathcal J_\eta \mathcal J_\eta \Psi }\|_{L^2}^2 + \frac{DN_A k_B T^*}2 \sum\limits_{i=1}^N \left\|\fr{\na c_i +\frac{e}{k_B}z_ic_i T^{-1} \na \mathcal J_\eta \mathcal J_\eta \Psi}{\sqrt{c_i}} \right\|_{L^2}^2 
    \\
    \leq & C\|T-T_r\|_{L^2}^2 + C\|c_i\|^{2}_{L^1}\| u\|^{2}_{L^2} =  C\|T-T_r\|_{L^2}^2 + C\overline{c_i(0)}^2\| u\|^{2}_{L^2}.
    \end{split}
\end{equation}
Using Gr\"onwall's inequality and \eqref{est:T}, and putting the superscript $\eta$ back, we have for any $t\in[0,\mathcal T]$ that
\begin{equation}\label{est:3}
    \begin{split}
        &{\varepsilon}\left\|\nabla\mathcal J_\eta \Psi^\eta(t)\right\|_{L^2}^2 +  \|u^\eta(t)\|_{L^2}^2 + 2k_BN_A{T^*} E^\eta(t) + \int_0^t \left(\frac{2D}{\varepsilon} \|\mathcal J_\eta \rho^\eta(s)\|_{L^2}^2 + \nu \|\na u^\eta(s)\|_{L^2}^2 \right) ds
    \\
    &\hspace{0.2in}+ \int_0^t \frac{D k_B N_A T^*}2 \sum\limits_{i=1}^N\|\nabla\sqrt{c_i^\eta(s)}\|_{L^2}^2 ds + \int_0^t\frac{De^2 N_A}{2k_B}\sum_{i=1}^N \left\|z_i \sqrt{c_i^\eta(s)} T^\eta(s)^{-\frac12} {\na \mathcal J_\eta \mathcal J_\eta \Psi^\eta(s) }\right\|_{L^2}^2 ds 
    \\
    &\hspace{0.2in}+ \int_0^t \frac{D k_B N_A T^*}2 \sum\limits_{i=1}^N \left\|\fr{\na c_i^\eta(s) +\frac{e}{k_B}z_ic_i^\eta(s) T^\eta(s)^{-1} \na \mathcal J_\eta \mathcal J_\eta \Psi^\eta(s)}{\sqrt{c_i^\eta(s)}} \right\|_{L^2}^2  ds
    \leq C, 
    \end{split}
\end{equation}
where $C$ on the right-hand side depends only on time $\mathcal T$ and initial conditions $\|\nabla\Psi_0\|_{L^2}$, $\|T_0-T_r\|_{L^2}$, $\overline{c_i(0)}$, $\|u_0\|_{L^2}$, and $E(0)$, and is independent of $\eta$. Here we have used Proposition \ref{fenchel} to obtain the control 
\begin{align*}
    E^\eta(0) = &\sum\limits_{i=1}^N \int_{\TT^3} c_i^\eta(0) \log\frac{c_i^\eta(0)}{\overline{c_i^\eta(0)}} = \sum\limits_{i=1}^N \int_{\TT^3} c_i^\eta(0) \log c_i^\eta(0) dx - \overline{c_i(0)} \log \overline{c_i(0)}
    \\
    \leq & \sum\limits_{i=1}^N \int_{\TT^3} c_i(0) \log c_i(0) dx - \overline{c_i(0)} \log \overline{c_i(0)} = \sum\limits_{i=1}^N \int_{\TT^3} c_i(0) \log\frac{c_i(0)}{\overline{c_i(0)}} = E(0).
\end{align*}

\noindent{\bf Step 5. Uniform bounds.}
Notice that $\|\sqrt{c_i^\eta}\|_{L^2}^2 = \|c_i^\eta\|_{L^1} = \overline{c_i^\eta} = \overline{c_i^\eta(0)} = \overline{c_i(0)}$ is a constant, \eqref{est:3} yields
\begin{align}\label{bdd-c-1}
    &\sqrt{c_i^\eta} \,\, \text{and}\, \, \nabla \mathcal J_\eta \Psi^\eta\,\,  \text{are uniformly bounded in} \,\,  L^\infty(0,\mathcal T; L^2)\cap L^2(0,\mathcal T; H^1),
    \\
    &u^\eta \,\,  \text{are uniformly bounded in} \,\,  L^\infty(0,\mathcal T; H)\cap L^2(0,\mathcal T; V).
\end{align}
For an arbitrary test function $\phi\in C^\infty$, by subsequent applications of H\"older's, the Sobolev, and the Ladyzhenskaya inequalities, we have the following estimates:
\begin{align}
    \langle c_i^\eta, \phi\rangle &= \int_{\mathbb T^3} \sqrt{c_i^\eta} \sqrt{c_i^\eta} \phi dx \leq \|\sqrt{c_i^\eta}\|_{L^3} \|\sqrt{c_i^\eta}\|_{L^6} \|\phi\|_{L^2} \leq C \|\sqrt{c_i^\eta}\|_{L^2}^{\frac12} \|\sqrt{c_i^\eta}\|_{H^1}^{\frac32} \|\phi\|_{L^2},\\
    \langle  c_i^\eta, \phi\rangle &= \int_{\mathbb T^3} \sqrt{c_i^\eta}  \sqrt{c_i^\eta} \phi dx \leq C\|\sqrt{c_i^\eta}\|_{L^3}^2 \|\phi\|_{L^3} \leq C  \|\sqrt{c_i^\eta}\|_{L^2}\|\sqrt{c_i^\eta}\|_{H^1} \|\phi\|_{L^3}, \\
    \langle \nabla c_i^\eta, \phi\rangle &= \int_{\mathbb T^3} 2\sqrt{c_i^\eta} \nabla \sqrt{c_i^\eta} \phi dx \leq C\|\nabla\sqrt{c_i^\eta}\|_{L^2} \|\sqrt{c_i^\eta}\|_{L^3} \|\phi\|_{L^6} \leq C \|\sqrt{c_i^\eta}\|_{L^2}^{\frac12} \|\sqrt{c_i^\eta}\|_{H^1}^{\frac32} \|\phi\|_{L^6}.
\end{align}
As $W^{1,\frac65}\hookrightarrow L^2$ in 3D, the above together with \eqref{bdd-c-1} yield that 
\begin{align}\label{bdd-c-2}
    c_i^\eta \,\,  \text{are uniformly bounded in} \,\,  L^2(0,\mathcal T; L^{\frac32})\cap L^\frac{4}3(0,\mathcal T; W^{1,\frac65}) .
\end{align}
Thanks to \eqref{eqn:npb-full-mo-2}, one also infers that 
\begin{align}
    &\rho^\eta  \,\,  \text{are uniformly bounded in} \,\,  L^2(0,\mathcal T; L^{\frac32}) \cap L^\frac43(0,\mathcal T; W^{1,\frac65}),
    \\
    &\Psi^\eta \,\,  \text{are uniformly bounded in} \,\,  L^2(0,\mathcal T; W^{2,\frac32}) \cap L^\frac43(0,\mathcal T; W^{3,\frac65}).
\end{align}
\end{proof}

Next, we address the uniform boundedness of the time derivatives of solutions to the mollified system.

\begin{prop}\label{prop:time-derivative} Suppose the same assumption on the initial conditions in Proposition \ref{prop:uni-est-sol} holds. Then the sequence of time derivatives $(\partial_t T^\eta, \partial_t u^\eta, \partial_t \sqrt{c_i^\eta+1})$ satisfy
    \begin{align}
        &\partial_t T^\eta \,\, \text{are uniformly bounded in} \,\, L^{\frac43}(0,\mathcal T; H^{-1}),
        \\
        &\partial_t u^\eta \,\, \text{are uniformly bounded in} \,\, L^{\frac43}(0,\mathcal T; \mathcal D(A)'),
        \\
        &\partial_t \sqrt{c_i^\eta +1} \,\, \text{are uniformly bounded in} \,\, L^1(0,\mathcal T; H^{-2}).
    \end{align}
\end{prop}

\begin{proof}
    {\bf Step 1. Bounds of $\partial_t T^\eta$.} Taking the $L^2$ inner product of equation \eqref{eqn:npb-full-mo-5} with an arbitrary test function $\phi\in L^{4}(0,\mathcal T; H^1)$ yields
    \begin{align}
        \left| \left\langle \partial_t T^\eta, \phi\right\rangle \right| \leq \left| \left\langle \mathcal J_\eta u^\eta\cdot\nabla T^\eta, \phi\right\rangle \right| + \kappa \left| \left\langle \nabla T^\eta, \nabla\phi\right\rangle \right|.
    \end{align}
Thanks to Proposition \ref{prop:uni-est-sol}, it follows that
\begin{align}
        \left| \left\langle \partial_t T^\eta, \phi\right\rangle \right| \leq C\|u^\eta\|_{L^2}^{\frac12} \|\nabla u^\eta\|_{L^2}^{\frac12} \|\nabla T^\eta\|_{L^2} \|\phi\|_{H^1} + C\|\nabla T^\eta\|_{L^2} \|\phi\|_{H^1} < \infty.
    \end{align}
Therefore,
\begin{align}
    \partial_t T^\eta \,\, \text{are uniformly bounded in} \,\, L^{\frac43}(0,\mathcal T; H^{-1}).
\end{align}

\smallskip

\noindent {\bf Step 2. Bounds of $\partial_t u^\eta$.} Considering $\phi\in L^{4}(0,\mathcal T; \mathcal D(A))$, and taking the inner product of equation \eqref{eqn:npb-full-mo-3} with $\phi$, we estimate $|\langle \partial_t u^\eta, \phi\rangle|$. A good control of the nonlinear term $\left| \left\langle \mathcal J_\eta u^\eta\cdot\nabla u^\eta, \phi\right\rangle \right|$ is obtained by following a similar argument to the analogous term in the   temperature evolution, and all the other linear terms are easily manageable. The remaining nonlinear term in $\rho^{\eta}$ can be estimated using H\"older's inequality and the Sobolev inequality followed by interpolation as shown below,
\begin{align}
    \left| \left\langle \mathcal J_\eta (\rho^\eta\nabla \mathcal J_\eta \mathcal J_\eta \Psi^\eta), \phi   \right\rangle\right|\leq C\|\rho^\eta\|_{L^2} \|\nabla \mathcal J_\eta \Psi^\eta\|_{L^2} \|\phi\|_{H^2}.
\end{align}
The latter is uniformly bounded thanks to Proposition \ref{prop:uni-est-sol}. Therefore,
\begin{align}
    \partial_t u^\eta \,\, \text{are uniformly bounded in} \,\, L^{\frac43}(0,\mathcal T; \mathcal D(A)').
\end{align}

\noindent{\bf Step 3. Bounds of $\partial_t c_i^\eta$.} As for the concentrations $c_i^\eta$, we would consider estimates for $\partial_t \sqrt{c_i^\eta+1}$ instead of $\partial_t c^\eta_i$. Indeed, the equation satisfied by $\partial_t \sqrt{c^\eta_i+1}$ is given by 
\begin{align}\label{eqn:sqrt-ci}
    \partial_t \sqrt{c^\eta_i+1} = \frac{\partial_t c_i^\eta}{2\sqrt{c^\eta_i+1}} = \frac{1}{2\sqrt{c^\eta_i+1}} \left(-\mathcal J_\eta  u^\eta\cdot \nabla c_i^\eta  + D\Delta c_i^\eta + D\frac{e}{k_B}\nabla\cdot (z_i\frac{1}{T^\eta} c_i^\eta \nabla \mathcal J_\eta \mathcal J_\eta \Psi^\eta)\right).
\end{align}
The reason for considering $\sqrt{c^\eta_i+1}$ instead of $\sqrt{c^\eta_i}$ is to avoid a vanishing denominator. Note that \eqref{bdd-c-1} implies that
\begin{align}\label{bdd-c-3}
    \sqrt{c_i^\eta+1} \,\,  \text{are uniformly bounded in} \,\,  L^\infty(0,\mathcal T; L^2)\cap L^2(0,\mathcal T; H^1).
\end{align}
We consider a test function $\phi\in L^{\infty}(0,\mathcal T; H^2)$. Taking the inner product of equation \eqref{eqn:sqrt-ci} with $\phi$, integrating by parts, and using H\"older's inequality, the Sobolev inequality, and interpolation inequalities, we estimate each term as follows. First for the diffusion term, one has
\begin{align}
    \left|\int_{\mathbb T^3} \frac{\Delta c_i^\eta}{\sqrt{c^\eta_i+1}} \phi  dx \right| =&  \left|-\int_{\mathbb T^3} \frac{\nabla c_i^\eta}{\sqrt{c^\eta_i+1}} \cdot \nabla\phi  dx  +  \int_{\mathbb T^3}  \frac{\nabla \sqrt{c^\eta_i+1}}{c^\eta_i+1} \cdot \nabla c_i^\eta \phi dx\right|
    \\
    &= \left|-2\int_{\mathbb T^3} \nabla\sqrt{c^\eta_i+1} \cdot \nabla\phi  dx  +  2\int_{\mathbb T^3}  \frac{|\nabla \sqrt{c^\eta_i+1}|^2}{\sqrt{c^\eta_i+1}} \phi dx\right|
    \\
    &\leq C \|\nabla\sqrt{c^\eta_i+1}\|_{L^2} \|\nabla\phi\|_{L^2} + C\|\nabla \sqrt{c^\eta_i+1}\|_{L^2}^2 \|\phi\|_{H^2} \underbrace{\|\frac1{\sqrt{c_i^\eta+1}}\|_{L^\infty}}_{\leq 1}
    \\
    &\leq C \|\nabla\sqrt{c^\eta_i+1}\|_{L^2} \|\phi\|_{H^1} + C\|\nabla \sqrt{c^\eta_i+1}\|_{L^2}^2 \|\phi\|_{H^2},
\end{align}
which together with \eqref{bdd-c-3} yield that $\frac{\Delta c_i^\eta}{\sqrt{c^\eta_i+1}}$ are uniformly bounded in $L^1(0,\mathcal T; H^{-2})$. As for the advection term, we have
\begin{align}
    \left|\int_{\mathbb T^3} \frac{\mathcal J_\eta u^\eta\cdot\nabla c_i^\eta}{2\sqrt{c^\eta_i+1}} \phi  dx \right| =& \left|\int_{\mathbb T^3} \mathcal J_\eta u^\eta\cdot\nabla \sqrt{c_i^\eta+1}\phi  dx \right| = \left|\int_{\mathbb T^3} \mathcal J_\eta u^\eta\cdot\nabla \phi \sqrt{c_i^\eta+1}  dx \right|
    \\
    \leq & \|\mathcal J_\eta u^\eta\|_{L^3} \|\nabla\phi\|_{L^6} \|\sqrt{c_i^\eta+1}\|_{L^2} \leq C \|u^\eta\|_{L^2}^{\frac12} \|u^\eta\|_{H^1}^{\frac12} \|\sqrt{c_i^\eta+1}\|_{L^2} \|\phi\|_{H^2},
\end{align}
implying that $\frac{\mathcal J_\eta u^\eta\cdot\nabla c_i^\eta}{2\sqrt{c^\eta_i+1}}$ are uniformly bounded in $L^4(0,\mathcal T; H^{-2})$. Regarding the electromigration term, it holds that 
\begin{align}
    &\left|\int_{\mathbb T^3} \frac{\nabla\cdot(\frac1{T^\eta}c_i^\eta\nabla \mathcal J_\eta  \mathcal J_\eta  \Psi^\eta)}{\sqrt{c^\eta_i+1}} \phi  dx \right| 
    \\
    \leq & \left|\int_{\mathbb T^3} \frac{\frac1{T^\eta}c_i^\eta\nabla \mathcal J_\eta  \mathcal J_\eta  \Psi^\eta}{\sqrt{c^\eta_i+1}} \cdot \nabla \phi  dx \right|  + \left|\int_{\mathbb T^3} \frac1{T^\eta}c_i^\eta \phi \nabla \mathcal J_\eta  \mathcal J_\eta  \Psi^\eta  \cdot \frac{\nabla \sqrt{c^\eta_i+1}}{c^\eta_i+1}   dx \right|
    \\
    \leq & \|\frac1{T^\eta}\|_{L^\infty} \|c_i^\eta\|_{L^2} \|\nabla \mathcal J_\eta  \mathcal J_\eta  \Psi^\eta\|_{L^3} \|\nabla\phi\|_{L^6} \|\frac1{\sqrt{c^\eta_i+1}}\|_{L^\infty} \\
    &\hspace{1cm}+\|\frac1{T^\eta}\|_{L^\infty} \|\phi\|_{L^\infty} \|\frac{c_i^\eta}{c^\eta_i+1}\|_{L^\infty} \|\nabla \mathcal J_\eta \Psi^\eta\|_{L^2} \|\nabla \sqrt{c^\eta_i+1}\|_{L^2}
    \\
    \leq & C\|c_i^\eta\|_{L^2} \|\nabla \mathcal J_\eta  \Psi^\eta\|_{L^2}^{\frac12} \|\nabla \mathcal J_\eta  \Psi^\eta\|_{H^1}^{\frac12} \|\phi\|_{H^2} + C \|\nabla \mathcal J_\eta  \Psi^\eta\|_{L^2}  \|\nabla \sqrt{c^\eta_i+1}\|_{L^2}\|\phi\|_{H^2}.
\end{align}
By virtue of Proposition \ref{prop:uni-est-sol} and the uniform boundedness \eqref{bdd-c-3}, it follows that $\frac{\nabla\cdot(\frac1{T^\eta}c_i^\eta\nabla \mathcal J_\eta  \mathcal J_\eta  \Psi^\eta)}{\sqrt{c^\eta_i+1}}$ are uniformly bounded in $L^1(0,\mathcal T; H^{-2})$. The three estimates above together imply that 
\begin{align}
    \partial_t \sqrt{c_i^\eta +1} \,\, \text{are uniformly bounded in} \,\, L^1(0,\mathcal T; H^{-2}).
\end{align}
\end{proof}

\begin{cor}\label{cor-limit}
Suppose the same assumption on the initial conditions in Proposition \ref{prop:uni-est-sol} holds. Then 
    there exists a decreasing sequence of numbers $\{\eta_k\}_{k\in\mathbb N}$ with $\lim\limits_{k\to \infty} \eta_k = 0$ such that the sequence of solutions $(u^{\eta_k}, c_i^{\eta_k}, T^{\eta_k})$ satisfy
   \begin{align}
        &u^{\eta_k} \rightharpoonup u \,\, \text{in} \,\, L^2(0,\mathcal T; V),\quad u^{\eta_k} \stackrel{\ast}{\rightharpoonup} u \,\, \text{in} \,\, L^\infty(0,\mathcal T; H), \quad u^{\eta_k}\rightarrow u \,\, \text{in} \,\, L^2(0,\mathcal T; H),
        \\
        &T^{\eta_k} \rightharpoonup T \,\, \text{in} \,\, L^2(0,\mathcal T; H^1),\quad T^{\eta_k} \stackrel{\ast}{\rightharpoonup} T \,\, \text{in} \,\, L^\infty(0,\mathcal T; L^p), \quad T^{\eta_k}\rightarrow T \,\, \text{in} \,\, L^2(0,\mathcal T; L^p),
        \\
        &\sqrt{c_i^{\eta_k}}  \rightharpoonup \sqrt{c_i} \,\, \text{in} \,\, L^2(0,\mathcal T; H^1),\quad \sqrt{c_i^{\eta_k}} \stackrel{\ast}{\rightharpoonup} \sqrt{c_i} \,\, \text{in} \,\, L^\infty(0,\mathcal T; L^2),
        \\
        &\sqrt{c_i^{\eta_k} +1} \to \sqrt{c_i+1} \,\, \text{and}\,\, \sqrt{c_i^{\eta_k}} \to \sqrt{c_i} \,\, \text{in} \,\, L^2(0,\mathcal T; L^2),
        \\
        &c_i^{\eta_k} \rightharpoonup c_i \,\, \text{in} \,\, 
L^2(0,\mathcal T; L^{\frac32}) \cap L^\frac43(0,\mathcal T;  W^{1,\frac65}). 
     \end{align}
\end{cor}
\begin{proof}
    The proof follows directly from Proposition \ref{prop:uni-est-sol}, Proposition \ref{prop:time-derivative}, the Banach-Alaoglu theorem, and the Aubin-Lions compactness theorem. 
\end{proof}

\section{Existence of Global Weak Solutions}\label{sec:existence}
In this section, we prove our first main result, Theorem \ref{thm:main-torus}, about the existence of global weak solutions to the original NPB system \eqref{sys:npb-full} in $\TT^3$.

\begin{proof}[Proof of Theorem \ref{thm:main-torus}]
    By virtue of Proposition \ref{thm:mollified-sys}, for each $k\in\mathbb N$ and $\eta_k>0$ we know that the mollified system \eqref{sys:npb-full-mo} has a unique global smooth solution $(c_i^{\eta_k}, u^{\eta_k}, T^{\eta_k})$. By Corollary \ref{cor-limit}, we know that they have corresponding limits $(c_i, u, T)$ with desired regularity as in \eqref{regularity-thm}. As $c_i^\ek(x,t)\geq 0$ and $T^\ek(x,t)\geq T^*$ for all $(x,t)\in \TT^3\times [0,\mathcal T]$, it follows that $c_i(x,t)\geq 0$ and $T(x,t)\geq T^*$ for a.e. $(x,t)\in\mathbb T^3\times [0,\mathcal T]$.
    In addition, from Proposition \ref{prop:uni-est-sol} and the Banach-Alaoglu theorem, we infer that there exists a limit function $\Psi$ such that $\Psi^{\eta_k} \rightharpoonup \Psi$ in $L^2(0,\mathcal T; W^{2,\frac32})\cap L^\frac43(0,\mathcal T; W^{3,\frac65})$. We define $\rho = \sum_{i=1}^N eN_Az_i c_i$ and $\widetilde\rho = -\varepsilon\Delta \Psi$. Then by uniqueness of limits we have $\rho=\widetilde\rho$, and therefore $ -\varepsilon\Delta \Psi= \rho =\sum_{i=1}^N eN_Az_i c_i$.
    
    For any time $\mathcal T>0$ and any test function $\phi,\varphi \in C^\infty([0,\mathcal T]\times \TT^3)$ such that $\phi(\mathcal T) = \varphi(\mathcal T)=0$ and $\nabla\cdot\varphi=0$, we take the space-time inner product of equations \eqref{eqn:npb-full-mo-1}, \eqref{eqn:npb-full-mo-3}, and \eqref{eqn:npb-full-mo-5} with $\phi$, $\varphi$, and $\phi$ respectively
    over $[0,\mathcal T]\times \TT^3$. We obtain
    \begin{subequations}
        \begin{align}
            &\int_0^{\mathcal T} \left[\langle\partial_t c_i^{\ek} ,\phi\rangle + \langle \mathcal J_{\ek}u^\ek\cdot\nabla c_i^\ek,\phi\rangle -D\langle\Delta c_i^\ek,\phi\rangle - D\frac{e}{k_B}\langle \nabla\cdot(z_i \frac1{T^\ek} c_i^\ek \nabla\mathcal J_\ek \mathcal J_\ek\Psi^\ek ),\phi\rangle\right] dt = 0,
            \\
            &\int_0^{\mathcal T} \Big[\langle\partial_t u^{\ek} ,\varphi\rangle + \langle \mathcal J_{\ek}u^\ek\cdot\nabla u^\ek,\varphi\rangle -\nu\langle\Delta u^\ek,\varphi\rangle \Big] dt 
            \\
            &\hspace{1cm} = \int_0^{\mathcal T} \Big[ g\Big\langle \alpha_T(T^\ek-T_r^\ek) - \alpha_S(\sum_{i=1}^N c_i^\ek M_i - S_r^\ek) ,\varphi_3\Big\rangle - \langle \mathcal J_\ek(\rho^\ek \nabla\mathcal J_\ek\mathcal J_\ek \Psi^\ek), \varphi\rangle  \Big] dt,
            \\
            &\int_0^{\mathcal T} \Big[\langle\partial_t T^{\ek} ,\phi\rangle + \langle \mathcal J_{\ek}u^\ek\cdot\nabla T^\ek,\phi\rangle -\kappa\langle\Delta T^\ek,\phi\rangle \Big] dt = 0,
        \end{align}
    \end{subequations}
    Our goal is to show that as $k\to \infty$ all the terms converge to the desired limits, and thus prove \eqref{thm:eqn1}--\eqref{thm:eqn3}.

\smallskip
\noindent{\bf Step 1. Convergence of linear terms.} We first consider the convergence of $\int_0^\mathcal T\langle \partial_t c_i^{\eta_k}, \phi\rangle dt$. Thanks to Corollary \ref{cor-limit} and by integration by parts,
we have
\begin{align}
    \int_0^\mathcal T\langle \partial_t c_i^{\eta_k}, \phi\rangle dt =  - \langle c_i^{\eta_k}(0), \phi(0)\rangle - \int_0^\mathcal T\langle  c_i^{\eta_k}, \partial_t\phi\rangle dt
    \to - \langle c_i(0), \phi(0)\rangle - \int_0^\mathcal T\langle  c_i, \partial_t\phi\rangle dt,
\end{align}
where we have used $c_i^{\eta_k}(0)\to c_i(0)$ in $L^2$ that follows from the convergence properties of mollifiers. Similarly, 
\begin{align*}
    \int_0^\mathcal T\langle \partial_t u^{\eta_k}, \varphi\rangle dt \to - \langle u_0, \varphi(0)\rangle - \int_0^\mathcal T\langle  u, \partial_t\varphi\rangle dt, \quad  \int_0^\mathcal T\langle \partial_t T^{\eta_k}, \phi\rangle dt \to - \langle T_0, \phi(0)\rangle - \int_0^\mathcal T\langle  T, \partial_t\phi\rangle dt.
\end{align*}
Integrating by parts and using the periodic boundary conditions, we have
\begin{align*}
     -\int_0^\mathcal T\langle \Delta c_i^{\eta_k}, \phi\rangle dt = \int_0^\mathcal T\langle  \nabla c_i^{\eta_k}, \nabla \phi\rangle dt \to \int_0^\mathcal T\langle \nabla c_i, \nabla \phi\rangle dt,
\end{align*}
thanks to Corollary \ref{cor-limit}. Analogously,
\begin{align*}
   -\int_0^\mathcal T\langle \Delta u^{\eta_k}, \varphi\rangle dt \to  \int_0^\mathcal T\langle \nabla u, \nabla\varphi\rangle dt, \quad -\int_0^\mathcal T\langle \Delta T^{\eta_k}, \phi\rangle dt \to  \int_0^\mathcal T\langle \nabla T, \nabla\phi\rangle dt.
\end{align*}
As $T_r^{\eta_k}=\overline{T^{\eta_k}} = \overline{T_0^{\eta_k}} = \overline{T_0}=T_r$ and $S_r^{\eta_k} = \sum\limits_{i=1}^N \overline{c_i^{\eta_k}} M_i =  \sum\limits_{i=1}^N\overline{c_i^{\eta_k}(0)} M_i  =  \sum\limits_{i=1}^N\overline{c_i(0)} M_i=S_r$, thanks to Corollary \ref{cor-limit}, one obtains the convergence of the linear term
\begin{align*}
    \int_0^{\mathcal T} \Big\langle \alpha_T (T^{\eta_k} - T_r^{\eta_k}) - \alpha_S(\sum\limits_{i=1}^N c_i^{\eta_k}M_i - S_r^{\eta_k}),\varphi_3\Big\rangle dt
    \to \int_0^{\mathcal T} \Big\langle \alpha_T (T - T_r) - \alpha_S(\sum\limits_{i=1}^N c_iM_i - S_r),\varphi_3\Big\rangle dt.
\end{align*}

\noindent{\bf Step 2. Convergence of advection terms.} For the advection terms in the concentration evolutions, for each $i\in\{1,\dots,N\}$ we bound
\begin{align}
    &\left|\langle \mathcal J_{\ek} u^{\ek} \cdot \nabla c_i^\ek - u\cdot\nabla c_i, \phi\rangle \right| 
    \\
    \leq &  \left|\langle \mathcal J_{\ek} (u^{\ek} - u) \cdot \nabla c_i^\ek, \phi\rangle \right| + \left|\langle (\mathcal J_{\ek} u - u) \cdot \nabla c_i^\ek, \phi\rangle \right| + \left|\langle u \cdot \nabla (c_i^\ek - c_i), \phi\rangle \right| 
    \\
    \leq & \left|\langle \mathcal J_{\ek} (u^{\ek} - u) \cdot \nabla \phi,  c_i^\ek \rangle \right| + \left|\langle (\mathcal J_{\ek} u - u) \cdot \nabla \phi, c_i^\ek \rangle \right| + \left|\langle u \cdot \nabla \phi, c_i^\ek - c_i\rangle \right| := A_1 + A_2 + A_3.
\end{align}
The first two terms are mainly controlled via interpolation as follows,
\begin{align}
    (A_1 + A_2)  \leq &  \|u^{\ek} - u\|_{L^3} \|\nabla\phi\|_{L^\infty} \|\sqrt{c_i^\ek}\|^2_{L^3}  +  \|\mathcal J_{\ek} u -u\|_{L^3} \|\nabla\phi\|_{L^\infty} \|\sqrt{c_i^\ek}\|^2_{L^3} 
    \\
    \leq & C\left( \|u^{\ek} - u\|_{L^2}^{\frac12} \|u^{\ek} - u\|_{H^1}^{\frac12} + \|\mathcal J_\ek  u - u\|_{L^2}^{\frac12} \|\mathcal J_\ek  u - u\|_{H^1}^{\frac12} \right) \|\nabla\phi\|_{L^\infty} \|\sqrt{c_i^\ek}\|_{L^2} \|\sqrt{c_i^\ek}\|_{H^1}.
\end{align}
As a consequence of Proposition \ref{prop:uni-est-sol} and Corollary \ref{cor-limit}, one has
\begin{align}
    &\int_0^{\mathcal T} (A_1 + A_2) dt  
    \\
    \leq &C\left( \|u^{\ek} - u\|_{L^2(0,\mathcal T;L^2)}^{\frac12} \|u^{\ek} - u\|_{L^2(0,\mathcal T;H^1)}^{\frac12} + \|\mathcal J_\ek  u - u\|_{L^2(0,\mathcal T;L^2)}^{\frac12} \|\mathcal J_\ek  u - u\|_{L^2(0,\mathcal T;H^1)}^{\frac12} \right) 
    \\
    &\hspace{1cm}\times \|\nabla\phi\|_{L^\infty(0,\mathcal T;L^\infty)} \|\sqrt{c_i^\ek}\|_{L^\infty(0,\mathcal T;L^2)} \|\sqrt{c_i^\ek}\|_{L^2(0,\mathcal T;H^1)} \to 0.
\end{align}
The latter convergence uses the mollifier property $\|\mathcal J_\ek  u - u\|_{L^2(0,\mathcal T;L^2)} \to 0$. As for $A_3$, we write 
\[
c_i^\ek-c_i= (\sqrt{c_i^\ek} - \sqrt{c_i})(\sqrt{c_i^\ek} + \sqrt{c_i}),
\]
and estimate
\begin{align}
    \int_0^{\mathcal T} A_3 dt \leq & C \int_0^{\mathcal T} \|u\|_{L^3} \|\nabla\phi\|_{L^\infty} \|\sqrt{c_i^\ek} - \sqrt{c_i}\|_{L^3} \|\sqrt{c_i^\ek} + \sqrt{c_i}\|_{L^3} dt
    \\
    \leq & C\|u\|_{L^\infty(0,\mathcal T; L^2)}^{\frac12} \|u\|_{L^2(0,\mathcal T; H^1)}^{\frac12} \|\sqrt{c_i^\ek} - \sqrt{c_i}\|_{L^2(0,\mathcal T; L^2)}^{\frac12} \|\sqrt{c_i^\ek} - \sqrt{c_i}\|_{L^2(0,\mathcal T; H^1)}^{\frac12}
    \\
    &\hspace{0.5cm}\times \|\sqrt{c_i^\ek} + \sqrt{c_i}\|_{L^\infty(0,\mathcal T; L^2)}^{\frac12} \|\sqrt{c_i^\ek} + \sqrt{c_i}\|_{L^2(0,\mathcal T; H^1)}^{\frac12} \|\nabla\phi\|_{L^\infty(0,\mathcal T;L^\infty)} \to 0.
\end{align}
Therefore, we deduce that
\begin{align}
    \int_0^{\mathcal T} \langle \mathcal J_\ek  u^\ek \cdot \nabla c_i^\ek ,\phi\rangle dt \to \int_0^{\mathcal T} \langle  u\cdot\nabla c_i,\phi\rangle dt.
\end{align}
The convergence of advection terms in the momentum equation and the temperature evolution follows similarly. Indeed the proof is simpler, as $u$ and $T$ have better bounds compare to $c_i$. It follows that
\begin{align}
    \int_0^{\mathcal T} \langle \mathcal J_\ek  u^\ek \cdot \nabla u^\ek ,\varphi\rangle dt \to \int_0^{\mathcal T} \langle  u\cdot\nabla u,\varphi\rangle dt, \quad 
    \int_0^{\mathcal T} \langle \mathcal J_\ek  u^\ek \cdot \nabla T^\ek ,\phi\rangle dt \to \int_0^{\mathcal T} \langle  u\cdot\nabla T,\phi\rangle dt.
\end{align}

\noindent{\bf Step 3. Convergence of the electromigration term.} Now we turn our attention to the convergence of the electromigration term that majorly differs from the other electrodiffusion models studied in the literature. This term in the NPB system  has a super nonlinear aspect governed by the variance of the temperature in space and time.  In order to overcome the challenges arising from this technical issue, we exploit the $L^{3+\delta}$ regularity imposed on the initial temperature (which propagates uniformly in time).  Indeed, we start by performing the following decompositon, 
\begin{align}
    &\left|\langle \nabla\cdot\Big(\frac1{T^\ek}c_i^\ek \nabla \mathcal J_\ek  \mathcal J_\ek  \Psi^\ek - \frac1T c_i \nabla\Psi\Big),\phi\rangle \right|
    \\
    =&\left|\langle \frac1{T^\ek}c_i^\ek \nabla \mathcal J_\ek  \mathcal J_\ek  \Psi^\ek - \frac1T c_i \nabla\Psi,\nabla\phi\rangle \right|
    \\
    \leq &\left|\langle (\frac1{T^\ek} - \frac1T) c_i^\ek \nabla \mathcal J_\ek  \mathcal J_\ek  \Psi^\ek ,\nabla\phi\rangle \right| 
    + 
    \left|\langle  \frac1T (c_i^\ek - c_i) \nabla \mathcal J_\ek  \mathcal J_\ek  \Psi^\ek ,\nabla\phi\rangle \right|
    \\
    &+
    \left|\langle  \frac1T c_i  (\mathcal J_\ek  \mathcal J_\ek  \nabla \Psi^\ek - \nabla \Psi) ,\nabla\phi\rangle \right|
    :=  B_1 + B_2 + B_3,
\end{align} and then we estimate each of $B_1$, $B_2$, and $B_3$ separetely. 

From \eqref{est:3} we have 
$
    \sqrt{c_i^\ek} \frac{1}{\sqrt{T^\ek}} \na \mathcal J_\ek \mathcal J_\ek \Psi^\ek  \in L^2(0,\mathcal T; L^2).
$
Given $\delta>0$, letting  $\delta_1 = \frac{4\delta}{1+\delta}$ and using H\"older's inequality, we bound $B_1$ as follows,
\begin{align}
    B_1 \leq & \left\|\frac1{\sqrt{T^\ek}}\right\|_{L^\infty} \left\|\frac1{T}\right\|_{L^\infty} \|T-T^\ek\|_{L^{3+\delta}}\|\sqrt{c_i^\ek}\|_{L^{6-\delta_1}} \|\sqrt{c_i^\ek}\frac1{\sqrt{T^\ek}} \nabla \mathcal J_\ek  \mathcal J_\ek  \Psi^\ek\|_{L^2} \|\nabla\phi\|_{L^\infty}.
\end{align}
By the Gagliardo–Nirenberg interpolation inequality, it follows that
\begin{align}
    \|\sqrt{c_i^\ek}\|_{L^{6-\delta_1}} \leq C \|\sqrt{c_i^\ek}\|_{H^1}^{\frac3{3+\delta}} \|\sqrt{c_i^\ek}\|_{L^2}^{\frac{\delta}{3+\delta}}.
\end{align}
Recall that $T^\ek\geq T^*$ and $T\geq T^*$, we have
\begin{align}\label{convergence-1}
    B_1 \leq C \|T-T^\ek\|_{L^{3+\delta}}^{\frac{2\delta}{6+2\delta}} \|T-T^\ek\|_{L^{3+\delta}}^{\frac{6}{6+2\delta}} \|\sqrt{c_i^\ek}\|_{H^1}^{\frac3{3+\delta}} \|\sqrt{c_i^\ek}\|_{L^2}^{\frac{\delta}{3+\delta}}\|\frac{\sqrt{c_i^\ek}}{\sqrt{T^\ek}} \nabla \mathcal J_\ek  \mathcal J_\ek  \Psi^\ek\|_{L^2}\|\nabla\phi\|_{L^\infty}.
\end{align}
From Corollary \ref{cor-limit}, we have the strong convergence
    $T^\ek \to T$ in $L^2(0,\mathcal T; L^{3+\delta})$ when $p=3+\delta$.
Integrating $B_1$ in time from $0$ to $\mathcal T$, and using H\"older's inequality in time, we manage to control the terms in \eqref{convergence-1} respectively by  their time Lebesgue $L^\frac{6+2\delta}{\delta}$, $L^\infty$, $L^{\frac{6+2\delta}{3}}$, $L^\infty$, $L^2$, and $L^\infty$  norms, and infer that
$\int_0^{\mathcal T} B_1 dt \to 0.$

A standard interpolation argument allows us to estimate $B_2$ as follows, 
\begin{align}
    B_2 \leq &\|\frac1T\|_{L^\infty} \|\sqrt{c_i^\ek}-\sqrt{c_i}\|_{L^3} \|\sqrt{c_i^\ek}+\sqrt{c_i}\|_{L^3} \|\nabla \mathcal J_\ek  \Psi^\ek\|_{L^3} \|\nabla \phi\|_{L^\infty}
    \\
    \leq & C\|\sqrt{c_i^\ek}-\sqrt{c_i}\|_{L^2}^{\frac12} \|\sqrt{c_i^\ek}-\sqrt{c_i}\|_{H^1}^{\frac12} \|\sqrt{c_i^\ek}+\sqrt{c_i}\|_{L^2}^{\frac12} \|\sqrt{c_i^\ek}+\sqrt{c_i}\|_{H^1}^{\frac12} 
    \\
    &\hspace{1cm}\times \|\nabla \mathcal J_\ek  \Psi^\ek\|_{L^2}^{\frac12}  \|\nabla \mathcal J_\ek  \Psi^\ek\|_{H^1}^{\frac12}  \|\nabla\phi\|_{L^\infty}.
\end{align}
Integrating in time from $0$ to $\mathcal{T}$, applying H\"older's inequality in time with exponents $4, 4, \infty, 4, \infty, 4, \infty$, and using Corollary  \ref{cor-limit}, we infer that $  \int_0^{\mathcal T} B_2 dt \to 0.$
    
    As for $B_3$, we observe that $\|\mathcal J_\ek  \mathcal J_\ek  \nabla \Psi^\ek\|_{L^2} \leq \|\mathcal J_\ek  \nabla \Psi^\ek\|_{L^2}$ and apply Proposition \ref{prop:uni-est-sol} to deduce that
$
\{\mathcal J_\ek  \mathcal J_\ek  \nabla \Psi^\ek\}_{k=1}^\infty$ are uniformly bounded in $L^\infty(0,\mathcal T; L^2)$, and thus in $L^4(0,\mathcal T; L^2)$.
Therefore, there exists a subsequence, also denoted by $\mathcal J_\ek  \mathcal J_\ek  \nabla \Psi^\ek$ (with a slight abuse of notation), that weakly converges to some function in $L^4(0,\mathcal T; L^2)$. We will show that this limit is actually $\nabla\Psi$. By elliptic estimate we have 
\begin{align}
    &\| \mathcal J_\ek  \mathcal J_\ek  \nabla \Psi^\ek -\nabla\Psi\|_{L^2}
    \\
    \leq &\| \mathcal J_\ek  \mathcal J_\ek  (\nabla \Psi^\ek - \nabla \Psi)\|_{L^2}
    + 
    \| \mathcal J_\ek   (\mathcal J_\ek  \nabla \Psi - \nabla \Psi)\|_{L^2} 
    +
    \| \mathcal J_\ek  \nabla \Psi - \nabla \Psi\|_{L^2} 
    \\
    \leq & \sum\limits_{i=1}^N C \|c_i^\ek- c_i\|_{L^\frac65} + 2\| \mathcal J_\ek  \nabla \Psi - \nabla \Psi\|_{L^2}.
\end{align}
The first term in the last line above can be controlled as follows, 
\begin{align}
     \|c_i^\ek- c_i\|_{L^\frac65} \leq &\|\sqrt{c_i^\ek} - \sqrt{c_i}\|_{L^2} \|\sqrt{c_i^\ek} + \sqrt{c_i}\|_{L^3} 
     \\
     \leq & C\|\sqrt{c_i^\ek} - \sqrt{c_i}\|_{L^2} \|\sqrt{c_i^\ek} + \sqrt{c_i}\|_{L^2}^{\frac12} \|\sqrt{c_i^\ek} + \sqrt{c_i}\|_{H^1}^{\frac12}.
\end{align}
By applying the  $L^{\frac43}$ norm in time to $\| \mathcal J_\ek  \mathcal J_\ek  \nabla \Psi^\ek -\nabla\Psi\|_{L^2}$, we obtain that 
\begin{align}
  &\| \mathcal J_\ek  \mathcal J_\ek  \nabla \Psi^\ek -\nabla\Psi\|_{L^{\frac43}(0,\mathcal T; L^2)} 
  \\
  \leq &C\sum\limits_{i=1}^N \|c_i^\ek- c_i\|_{L^{\frac43}(0,\mathcal T; L^\frac65)} + 2 \| \mathcal J_\ek  \nabla \Psi - \nabla \Psi\|_{L^{\frac43}(0,\mathcal T; L^2)}
  \\
  \leq &C  \|\sqrt{c_i^\ek} - \sqrt{c_i}\|_{L^2(0,\mathcal T; L^2)} \|\sqrt{c_i^\ek} + \sqrt{c_i}\|_{L^\infty(0,\mathcal T; L^2)}^{\frac12} \|\sqrt{c_i^\ek} + \sqrt{c_i}\|_{L^2(0,\mathcal T; H^1)}^{\frac12} 
  \\
  &\hspace{1cm}+ 2 \| \mathcal J_\ek  \nabla \Psi - \nabla \Psi\|_{L^{\frac43}(0,\mathcal T; L^2)} \to 0,
\end{align}
after making use of Proposition \ref{prop:uni-est-sol}, Corollary \ref{cor-limit}, and convergence properties of mollifiers. 
Then by the uniqueness of limits, we infer that $\mathcal J_\ek  \mathcal J_\ek  \nabla \Psi^\ek \rightharpoonup \nabla \Psi$ weakly in $L^4(0,\mathcal T; L^2)$.
Since $\frac1T, \nabla\phi\in L^\infty(0,\mathcal T; L^\infty)$ and $c_i\in L^{\frac43}(0,\mathcal T; L^2)$, we exploit the weak convergence of $\mathcal J_\ek  \mathcal J_\ek  \na \Psi^\ek$ to $\nabla \Psi$ to conclude that
$
    \int_0^{\mathcal T} B_3 dt \to 0.
$
Combining the convergence results for  $B_1$, $B_2$, and $B_3$ yields 
\begin{align}
     \int_0^{\mathcal T} \langle \frac1{T^\ek}c_i^\ek \nabla \mathcal J_\ek  \mathcal J_\ek  \Psi^\ek ,\nabla\phi\rangle dt \to \int_0^{\mathcal T} \langle \frac1T c_i \nabla\Psi ,\nabla\phi\rangle dt.
\end{align}

\noindent{\bf Step 4. Convergence of the electric forcing term.} Finally, we prove the convergence of the electric term forcing the Navier-Stokes equations. We have
\begin{align}
    &\left|\langle \mathcal J_\ek  (\rho^\ek \mathcal J_\ek  \mathcal J_\ek  \nabla\Psi^\ek) - \rho \nabla\Psi, \varphi  \rangle \right|
    \\
    \leq &\left|\langle \mathcal J_\ek  ((\rho^\ek-\rho) \mathcal J_\ek  \mathcal J_\ek  \nabla\Psi^\ek), \varphi  \rangle \right| + \left|\langle \mathcal J_\ek  (\rho \mathcal J_\ek  \mathcal J_\ek  \nabla\Psi^\ek  ) - \rho \mathcal J_\ek  \mathcal J_\ek  \nabla\Psi^\ek, \varphi  \rangle \right|
    \\
    &+ \left|\langle\rho \mathcal J_\ek  \mathcal J_\ek  \nabla\Psi^\ek - \rho \nabla\Psi, \varphi  \rangle \right| := C_1 + C_2 + C_3.
\end{align}

For $C_1$, by H\"older's inequality and the interpolation inequality, one deduces
\begin{align}
    C_1 &\leq C\|\rho^\ek - \rho\|_{L^{\frac32}} \| \mathcal J_\ek  \nabla\Psi^\ek\|_{L^3} \|\varphi\|_{L^\infty}
    \\
    &\leq C \sum\limits_{i=1}^N  \|\sqrt{c_i^\ek}-\sqrt{c_i}\|_{L^3} \|\sqrt{c_i^\ek}+\sqrt{c_i}\|_{L^3} \|\nabla \mathcal J_\ek  \Psi^\ek\|_{L^3} \|\varphi\|_{L^\infty}
    \\
    &\leq C \sum\limits_{i=1}^N \|\sqrt{c_i^\ek}-\sqrt{c_i}\|_{L^2}^{\frac12} \|\sqrt{c_i^\ek}-\sqrt{c_i}\|_{H^1}^{\frac12} \|\sqrt{c_i^\ek}+\sqrt{c_i}\|_{L^2}^{\frac12} \|\sqrt{c_i^\ek}+\sqrt{c_i}\|_{H^1}^{\frac12} 
    \\
    &\hspace{1cm}\times \|\nabla \mathcal J_\ek  \Psi^\ek\|_{L^2}^{\frac12}  \|\nabla \mathcal J_\ek  \Psi^\ek\|_{H^1}^{\frac12}  \|\varphi\|_{L^\infty}.
\end{align}
As shown for $B_2$, we deduce that 
$
    \int_0^{\mathcal T} C_1 dt \to 0.
$

For $C_2$, we compute
\begin{align}
    \int_0^{\mathcal T} C_2 dt \leq &C\|\rho\|_{L^{\frac43}(0,\mathcal T; L^2)}  \|\mathcal J_\ek  \nabla\Psi^\ek\|_{L^4(0,\mathcal T; L^3)} \|\mathcal J_\ek  \varphi - \varphi\|_{L^\infty(0,\mathcal T; L^6)}
    \\
    \leq &C\sum\limits_{i=1}^N \|c_i\|_{L^{\frac43}(0,\mathcal T; L^2)}  \|\mathcal J_\ek  \nabla\Psi^\ek\|^{\frac12}_{L^\infty(0,\mathcal T; L^2)} \|\mathcal J_\ek  \nabla\Psi^\ek\|_{L^2(0,\mathcal T; H^1)}^{\frac12} 
    \\
    &\hspace{2cm}\times\|\mathcal J_\ek  \varphi - \varphi\|_{L^\infty(0,\mathcal T; L^6)} \to 0,
\end{align}
where we have used the mollifier properties and Proposition \ref{prop:uni-est-sol}.

Finally, thanks to the weak convergence of $\mathcal J_\ek  \mathcal J_\ek  \nabla\Psi^\ek \rightharpoonup \nabla \Psi$ in $L^4(0,\mathcal T; L^2)$, as well as the regularity $\rho\in L^{\frac43}(0,\mathcal T; L^2)$ and $\varphi\in L^\infty(0,\mathcal T; L^\infty)$, we obtain $\int_0^{\mathcal T} C_3 dt \to 0$, as shown for the term $B_3$.

Combining the aforementioned results for $C_1$, $C_2$, and $C_3$, we conclude that 
\begin{align}
     \int_0^{\mathcal T} \langle \mathcal J_\ek  (\rho^\ek \mathcal J_\ek  \mathcal J_\ek  \nabla\Psi^\ek),\varphi\rangle dt \to  \int_0^{\mathcal T} \langle \rho\nabla\Psi,\varphi \rangle dt.
\end{align}

Therefore, $(c_i,u,T)$ is a weak solution to system \eqref{sys:npb-full} on $[0,\mathcal T]$ in the sense of \eqref{thm:eqn1}--\eqref{thm:eqn3} and satisfying the regularity \eqref{regularity-thm}. This finishes the proof.
\end{proof}

\section{Long-time Dynamics}\label{sec:longtime}
Given a weak solution $(u,c_i, T)$ to the NPB model obtained by Theorem \ref{thm:main-torus}, we study its long-time dynamics and establish its convergence in time to the steady state $(u^\#, c_i^\#, T^\#) = (0, \bar{c}_i, \bar{T}) = (0, \overline{c_i(0)}, \overline{T_0})$. In addition, we address the decay of the entropy $\int_{\TT^3} c_i \log \frac{c_i}{\bar{c}_i} dx$ to $0$.

\subsection{Convergence of entropies} In this subsection, we address the convergence of sequences of entropies in Lebesgue spaces. 
\beg{prop} \label{entropyconvergence} Let $\left\{f_n(t,x) \right\}_{n=1}^{\infty} $ be a sequence of nonnegative real-valued functions. Suppose the following holds:
\begin{enumerate}
    \item The sequence $\left\{\sqrt{f_n}\right\}_{n=1}^{\infty}$ converges in $L^2(0,\mathcal{T}; L^2)$ to $\sqrt{f}$ for some nonnegative function $f$;
    \item The sequence $\left\{f_n\right\}_{n=1}^{\infty}$ is uniformly bounded in $L^{\fr{3}{2}}(0,\mathcal{T}; L^{\fr{3}{2}})$, and $f \in L^{\fr{3}{2}}(0,\mathcal{T}; L^{\fr{3}{2}})$.
\end{enumerate} For $t \in [0, \mathcal{T}]$, define the sequence of entropies
$ 
\mathcal{E}_n(t) = \int_{\TT^3} f_n \log f_n dx.
$ Then $\left\{\mathcal{E}_n(t)\right\}_{n=1}^{\infty}$has a subsequence that converges to $\mathcal{E}(t)= \int_{\TT^3} f \log f dx$ in $L^1(0, \mathcal{T})$.
\end{prop}

\beg{proof}
For $n \in \N$ and $t \in [0, \mathcal{T}]$, we consider the remainders 
\be 
R_n(t) = \int_{\TT^3} |f_n \log f_n - f \log f | dx
\ee and bound $R_n(t)$ by the sum of three terms, 
$R_n^1(t)$, $ R_n^2(t)$, and  $R_n^3(t)$, where 
\be 
R_n^1(t) = \int_{\TT^3} |\sqrt{f_n} \sqrt{f_n} \log f_n - \sqrt{f_n} \sqrt{f} \log f_n| dx,
\ee 
\be 
R_n^2(t) =  \int_{\TT^3} |\sqrt{f_n} \sqrt{f} \log f_n - \sqrt{f_n} \sqrt{f} \log f| dx,
\ee 
\be 
R_n^3(t) =  \int_{\TT^3} |\sqrt{f_n} \sqrt{f} \log f - \sqrt{f} \sqrt{f} \log f| dx.
\ee By the time-space Cauchy-Schwarz inequality and the algebraic estimate $x (\log x)^2 \le 1+x^{\fr{3}{2}}$ that holds for nonnegative real numbers $x \in \R$, we estimate 
\be 
\beg{aligned}
\left|\int_{0}^{\mathcal{T}} R_n^1(t) dt\right| 
&\le \left(\int_{0}^{\mathcal{T}} \int_{\TT^3} \left(\sqrt{f_n} - \sqrt{f}\right)^2 dxdt\right)^{\fr{1}{2}} \left(\int_{0}^{\mathcal{T}} \int_{\TT^3} \left(\sqrt{f_n} \log f_n \right)^2  dxdt\right)^{\fr{1}{2}}
\\&\le C\left(\int_{0}^{\mathcal{T}} \int_{\TT^3} \left(\sqrt{f_n} - \sqrt{f}\right)^2 dxdt\right)^{\fr{1}{2}} \left(\int_{0}^{\mathcal{T}} \int_{\TT^3} \left(1 + f_n^{\fr{3}{2}} \right)  dxdt\right)^{\fr{1}{2}},
\end{aligned}
\ee where the right-hand side converges to $0$ as $n \rightarrow \infty$ due to the uniform boundedness of the functions $f_n$ in $L^{\fr{3}{2}}(0, \mathcal{T}; L^{\fr{3}{2}})$ and the convergence $\sqrt{f_n} \rightarrow \sqrt{f}$ in $L^2(0,\mathcal{T}; L^2)$. Similarly, $\int_{0}^{\mathcal{T}} R_n^3(t) dt$ converges to $0$. 

As $\left\{\sqrt{f_{n}}\right\}_{n=1}^{\infty}$ converges to $\sqrt{f}$ in $L^{2}(0, \mathcal{T}; L^2)$, we infer the existence of a subsequence $\left\{\sqrt{f_{n_{k}}}\right\}_{k=1}^{\infty}$ that converge to $\sqrt{f}$ a.e. on $[0,\mathcal{T}] \times \TT^3$. Consequently, $\left\{f_{n_{k}}\right\}_{k=1}^{\infty}$ converges to $f$ a.e. on $[0, \mathcal{T}] \times \TT^3$. 
For $k \in \N$, we define the sequences of functions
\be 
h_{n_k} = |\sqrt{f_{n_k}} \sqrt{f} \log f_{n_k} - \sqrt{f_{n_k}} \sqrt{f} \log f|, \quad
g_{n_{k}} = \sqrt{f} \left(1 + f_{n_{k}}^{\fr{3}{4}}\right) + \sqrt{f_{n_{k}}} \left( 1+ f^{\fr{3}{4}}\right).
\ee We note that $|h_{n_{k}}| \le g_{n_{k}}$ for all $k \in \N$, all $t \in [0,\mathcal{T}]$, and all $x \in \TT^3$. Moreover, $\left\{h_{n_{k}}\right\}_{k=1}^{\infty}$ and $\left\{g_{n_{k}}\right\}_{k=1}^{\infty}$ converge respectively to 0 and $2\sqrt{f}(1+f^{\fr{3}{4}})$ a.e. on $[0, \mathcal{T}] \times \TT^3$. We will show that 
\be \la{conver11}
\int_{[0, \mathcal{T}] \times \TT^3}  \sqrt{f} \left(1 + f_{n_{k}}^{\fr{3}{4}}\right) dxdt \rightarrow  \int_{[0, \mathcal{T}] \times \TT^3} \sqrt{f}(1+f^{\fr{3}{4}})dxdt, 
\ee
\be \la{conver12}
\int_{[0, \mathcal{T}] \times \TT^3} \sqrt{f_{n_{k}}} \left( 1+ {f}^{\fr{3}{4}}\right) dxdt \rightarrow \int_{[0, \mathcal{T}] \times \TT^3} \sqrt{f}(1+f^{\fr{3}{4}}) dxdt
\ee as $k \rightarrow \infty$ and deduce from the Generalized Dominated Convergence Theorem that $\left\{\int_{0}^{\mathcal{T}}R^2_{n_{k}} dt \right\}_{k=1}^{\infty}$ converges to 0. In order to prove \eqref{conver11},  we decompose the difference 
$
D_k = \int_{0}^{\mathcal{T}} \int_{\TT^3} [\sqrt{f} (f_{n_k}^{\fr{3}{4}} - f^{\fr{3}{4}} ) ] dxdt $ into the sum of two terms, $D_k^1$ and $D_k^2$, where
\be 
D_k^1 := \int_{0}^{\mathcal{T}} \int_{\TT^3} \sqrt{f} f_{n_k}^{\fr{1}{4}} \left(\sqrt{f_{n_k}} - \sqrt{f} \right) dxdt \quad \text{and} \quad 
D_k^2 :=  \int_{0}^{\mathcal{T}} \int_{\TT^3} \sqrt{f} \sqrt{f} \left(f_{n_k}^{\fr{1}{4}} - f^{\fr{1}{4}} \right) dxdt,
\ee and we prove that $\left\{D_k^1\right\}_{k=1}^{\infty}$ and $\left\{D_k^2\right\}_{k=1}^{\infty}$ converge to $0$ as $k \rightarrow \infty$. Indeed, 
\be 
|D_k^1| \le \left(\int_{0}^{\mathcal{T}} \int_{\TT^3} (\sqrt{f_{n_k}} - \sqrt{f})^2 dxdt  \right)^{\frac{1}{2}} \left(\int_{0}^{\mathcal{T}} \int_{\TT^3} f^{\fr{3}{2}}  dxdt \right)^{\frac{1}{3}} \left(\int_{0}^{\mathcal{T}} \int_{\TT^3} f_{n_k}^{\fr{3}{2}}  dxdt\right)^{\frac{1}{6}}
\ee in view of H\"older's inequality, and so $D_k^1 \rightarrow 0$. By making use of the algebraic identity $(a^{\fr{1}{4}} - b^{\fr{1}{4}})(a^{\fr{1}{4}} + b^{\fr{1}{4}}) = \sqrt{a} - \sqrt{b}$ that holds for any nonnegative real numbers $a$ and $b$, we estimate
\be 
\beg{aligned}
|D_k^2| 
&= \left|\int_{0}^{\mathcal{T}} \int_{\TT^3} \sqrt{f} \sqrt{f} \frac{\sqrt{f_{n_k}} - \sqrt{f}}{f_{n_k}^{\fr{1}{4}} + f^{\fr{1}{4}}} dxdt\right|
\le \int_{0}^{\mathcal{T}} \int_{\TT^3} \sqrt{f} \sqrt{f} \frac{|\sqrt{f_{n_k}} - \sqrt{f}|}{ f^{\fr{1}{4}}} dxdt
\\&= \int_{0}^{\mathcal{T}} \int_{\TT^3} f^{\fr{3}{4}} |\sqrt{f_{n_k}} - \sqrt{f}|  dxdt
\le \left(\int_{0}^{\mathcal{T}} \int_{\TT^3} (\sqrt{f_{n_k}} - \sqrt{f})^2 dxdt  \right)^{\frac{1}{2}} \left(\int_{0}^{\mathcal{T}} \int_{\TT^3} f^{\fr{3}{2}}  dxdt \right)^{\frac{1}{2}},
\end{aligned}
\ee and thus $\left\{D_k^2\right\}_{k=1}^{\infty}$ converges to $0$ as well. This gives the desired convergence property \eqref{conver11}. The proof of \eqref{conver12} is similar and thus is omitted. Therefore, $\left\{\int_{0}^{\mathcal{T}} R_{n_{k}} dt\right\}_{k=1}^{\infty}$ converges to 0. This ends the proof of Proposition \ref{entropyconvergence}.
\end{proof}

\begin{cor}\label{cor-entropy}
    Given $\mathcal T>0$, the sequence $E^\eta(t) = \sum_{i=1}^N \int_{\TT^3} c_i^\eta(t) \log\frac{c_i^\eta(t)}{\overline{c_i^\eta}} dx$ has a subsequence that converges to  $E(t)=\sum_{i=1}^N \int_{\TT^3} c_i(t) \log\frac{c_i(t)}{\bar{c}_i} dx$ strongly in $L^1(0,\mathcal T)$.
\end{cor}
\begin{proof}
    First notice that $$\int_{\TT^3} c_i^\eta(t) \log\overline{c_i^\eta} dx = \log\overline{c_i^\eta} \int_{\TT^3} c_i^\eta(t) dx= \log\overline{c_i} \int_{\TT^3} c_i(t) dx= \int_{\TT^3} c_i(t) \log\overline{c_i} dx.$$
    Then the result follows from Proposition \ref{prop:uni-est-sol}, Corollary \ref{cor-limit}, and Proposition \ref{entropyconvergence}.
\end{proof}

\subsection{Long-time dynamics} Now we present the proof of Theorem \ref{thm:longtime}.

\begin{proof}[Proof of Theorem \ref{thm:longtime}]
First, recall that from \eqref{est:T} we have $\|T^\eta(t)-\overline{T^\eta_0}\|_{L^2}^2 \leq \|T_0-\overline{T_0}\|_{L^2}^2 e^{-\frac{2\kappa}{C_p} t}$ for all $\eta$. If $T$ is a weak solution obtained by passing to the limit $\eta\to0$ from $T^\eta$, we know from Corollary \ref{cor-limit} that $\|T^\eta(t)-\overline{T^\eta_0}\|_{L^2} \to \|T(t)-\overline{T_0}\|_{L^2}$ strongly in $L^2(0,\mathcal T)$, thus a subsequence converges for a.e. $t\in[0,\mathcal T]$. Then it follows that for a.e. $t\in[0,\mathcal T]$,
\begin{align*}
   \|T(t)-\overline{T_0}\|_{L^2}^2 \leq \|T_0-\overline{T_0}\|_{L^2}^2 e^{-\frac{2\kappa}{C_p} t}. 
\end{align*}

Next, we address the decay in time and rate of convergence of the energy 
\begin{equation}
     \mathcal E^\eta:={\varepsilon}\left\|\nabla\mathcal J_\eta \Psi^\eta\right\|_{L^2}^2 +  \|u^\eta\|_{L^2}^2 + 2k_BN_A{T^*} E^\eta
\end{equation}
by exploiting the dissipation 
\begin{equation}
     \mathcal D^\eta:=\frac{2D}{\varepsilon} \|\mathcal J_\eta \rho^\eta\|_{L^2}^2 + \nu \|\na u^\eta\|_{L^2}^2 + \frac{Dk_BN_AT^*}2 \sum\limits_{i=1}^N\|\nabla \sqrt{c_i^\eta}\|_{L^2}^2
\end{equation}
that governs the time evolution of $\mathcal E^{\eta}$, as shown in \eqref{est:6}. In view of the Poisson equation $-\varepsilon \Delta \mathcal J_\eta \Psi^\eta = \mathcal J_\eta \rho^\eta$, we have the elliptic estimate
\begin{equation}
    \frac{2D}{\varepsilon}\|\mathcal J_\eta \rho^\eta\|_{L^2}^2 \geq \frac{2D\varepsilon}{C_e} \|\nabla\mathcal J_\eta \Psi^\eta\|_{L^2}^2,
\end{equation} 
and by the Poincar\'e inequality, one has 
\begin{equation}
    \frac\nu2\|\nabla u^\eta\|_{L^2}^2 \geq \frac{\nu}{2C_p} \|u^\eta\|_{L^2}^2.
\end{equation}
Applying the logarithmic Sobolev inequality \cite[Theorem 1]{abdo2024logarithmic} we obtain 
\begin{equation}
    \frac{Dk_BN_AT^*}4 \sum\limits_{i=1}^N\|\nabla \sqrt{c_i^\eta}\|_{L^2}^2 \geq \frac{Dk_B N_A T^*}{4C_s} \sum\limits_{i=1}^N \int_{\TT^3} c_i^\eta \log \fr{c_i^\eta}{\bar{c}_i} dx =  \frac{Dk_BN_AT^*}{4C_s} E^\eta. 
\end{equation}

Now notice that the term $\overline{c_i(0)}^2 \|u\|_{L^2}^2$ on the right-hand side of \eqref{est:6} tremendously destroys the decay of energy. Instead, an application of the Poincar\'e inequality allows this term to get absorbed by the dissipation under some size condition imposed on $\overline{c_i(0)}$. In order to track such a condition, we revisit \eqref{est:salinity} and obtain
\begin{equation} \la{correctedsal}
    \begin{split}
        &\left|\int_{\mathbb T^3} g\alpha_S(\sum\limits_{i=1}^N c_i^\eta M_i - S^\eta_r) u^\eta_3 dx\right| \leq  C\alpha_S\sum\limits_{i=1}^N \|c^\eta_i - \overline{c_i^\eta}\|_{L^{\frac32}} \|u^\eta\|^{\frac12}_{L^2}\|\nabla u^\eta\|^{\frac12}_{L^2}
        \\
    \leq & C\alpha_S\sum\limits_{i=1}^N \|c^\eta_i - \overline{c_i^\eta}\|_{W^{1,1}}  \|\nabla u^\eta\|_{L^2}  \leq \alpha_S\sum\limits_{i=1}^N \|\nabla c_i^\eta\|_{L^1}  \|\nabla u^\eta\|_{L^2}
    \\
    \leq & C\alpha_S\sum\limits_{i=1}^N \|c^\eta_i\|_{L^2} \|\nabla \sqrt{c^\eta_i}\|_{L^2}  \|\nabla u^\eta\|_{L^2} 
    \leq C\alpha_S\sum\limits_{i=1}^N \|c_i^\eta\|_{L^1}^{\frac12} \|\nabla \sqrt{c_i^\eta}\|_{L^2}  \|\nabla u^\eta\|_{L^2}
    \\
    \leq &\frac\nu4\|\nabla u^\eta\|^2_{L^2} + C\nu^{-1}\alpha_S^2 \sum\limits_{i=1}^N \overline{c_i(0)} \|\nabla \sqrt{c_i^\eta}\|^2_{L^2}, 
    \end{split}
\end{equation}
where the constant $C$ depends only on the domain $\TT^3$.
If we impose the following size assumption on $\overline{c_i(0)}$:
\begin{equation}\label{smallness}
    C\nu^{-1}\alpha_S^2 \max\limits_{i\in\{1,\dots,N\}} \overline{c_i(0)} \leq \frac{Dk_B N_AT^*}4 \Leftrightarrow  \max\limits_{i\in\{1,\dots,N\}} \overline{c_i(0)} \leq \frac{Dk_B N_AT^* \nu}{4C\alpha_S^2},
\end{equation}
we can use the dissipation to absorb the terms on the right-hand side of \eqref{correctedsal}.

Denoting by $M:= \min\{\frac{2D}{C_e}, \frac{\nu}{2C_p}, \frac{D}{8C_s}\}$,
 combining the estimates above, and taking advantage of the temperature decaying estimates \eqref{est:T}, we rewrite \eqref{est:6} as
\begin{equation}\label{est:7}
    \begin{split}
         \frac{d}{dt}\mathcal E^\eta + M\mathcal E^\eta
    \leq  C\|T^\eta-T_r\|_{L^2}^2 \leq C\|T^\eta_0-T_r\|_{L^2}^2 e^{-\frac{2\kappa}{C_p} t}.
    \end{split}
\end{equation}
Letting $\widetilde{M}:=\min\{M, \frac{\kappa}{C_p}\}$, we end up with the energy inequality 
\begin{equation}
    \frac{d}{dt}\mathcal E^\eta + \widetilde M\mathcal E^\eta\leq C\|T^\eta_0-T_r\|_{L^2}^2 e^{-2\widetilde M t}.
\end{equation}
Integrating both sides from $0$ to $t\in [0,\mathcal T]$, thanks to Proposition \ref{fenchel}, we conclude that for any $t\in [0,\mathcal T]$,
\begin{equation}\label{eqn:sec5-1}
    \mathcal E^\eta(t) \leq  \left(\mathcal E^\eta(0) + C\widetilde{M}^{-1}\|T^\eta_0-T_r\|_{L^2}^2\right) e^{-\widetilde Mt} \le \left(\mathcal E(0) + C\widetilde{M}^{-1}\|T_0-T_r\|_{L^2}^2\right) e^{-\widetilde Mt}.
\end{equation} 

If $(c_i,u,T)$ is a weak solution obtained by passing to the limit $\eta\to 0$ from $(c^\eta_i,u^\eta,T^\eta)$, we infer from Corollary \ref{cor-limit} and Corollary \ref{cor-entropy} that up to a subsequence, $\|u^\eta(t)\|_{L^2} \to  \|u(t)\|_{L^2}$ in $L^2(0,\mathcal T)$ and $E^\eta(t) \to E(t)$ in $L^1(0,\mathcal T)$. By passing to a further subsequence in $\eta\to 0$, \eqref{eqn:sec5-1} implies that for a.e. $t\in [0,\mathcal T]$,
\begin{equation}
    \|u(t)\|_{L^2}^2 + E(t) = \|u(t)\|_{L^2}^2 + \int_{\TT^3} c_i(t) \log\frac{c_i(t)}{\bar{c}_i} \le  Ce^{-\widetilde Mt}.
\end{equation}

Finally, thanks to the Csiszar-Kullback-Pinsker inequality (see e.g. \cite{Tsybakov2008}), we have $\|c_i(t)-\overline{c_i(0)}\|^2_{L^1} = \|c_i(t)-\overline{c_i}\|^2_{L^1} \leq  CE(t)$, which leads to the exponential decay in time of $c_i(t)$ to $\overline{c_i(0)}$ in $L^1$. 
\end{proof}

\begin{rem}
 We point out that the size condition \eqref{smallness} can be dropped in the case of $\alpha_S=0$, i.e., when the salinity in the buoyancy force is negligible. Moreover, as $k_B N_A=R \approx 8.3$ is of order $O(1)$, the size condition \eqref{smallness} makes sense in real-world applications.
\end{rem}

\section{Concluding Remarks}\label{sec:conclusion}
We introduce a new electrodiffusion model, the Nernst-Planck-Boussinesq (NPB) system, which incorporates variational temperature and buoyancy forces, enhancing its generality and realism compared to other electrodiffusion models which assume constant temperature in space and time. From a mathematical point of view, temperature variation in the NPB system brings forth a more complicated nonlinear structure in the electromigration term. Due to the variation of temperature, the techniques employed for the NPNS system (see, for example, \cite{constantin2019nernst}) are not applicable. 

In the case of equal diffusivities for all ionic species, we successfully establish the global existence of weak solutions to the NPB system with periodic boundary conditions in 3D, with initial conditions $\nabla\Psi_0\in L^2$, $c_i(0)\in L^1$, $c_i(0)\log c_i(0)\in L^1$, $u_0\in H$, $T_0\in L^{3+\delta}$ for any arbitrary $\delta>0$, and under the assumptions that $c_i(0)$ are nonnegative and $T_0\geq T^*>0$ (which follows from the Third Law of Thermodynamics). The slightly higher regularity requirement for $T_0$ is due to the cubic nonlinearity in the electromigration term. Additionally,  we study the long-time dynamics of weak solutions and establish the exponential decay in time of $(c_i(t), u(t), T(t))$ to $(\overline{c_i(0)}, 0, \overline{T_0})$ and of the relative entropy $\int_{\TT^3} c_i(t) \log\frac{c_i(t)}{\bar{c}_i}$ to $0$ under some size constraint on the initial condition $\overline{c_i(0)}$. 

Many interesting open problems need further investigation for the NPB system. One such problem is the existence of global weak solutions where the ionic species have different diffusivities. Another natural question is whether the 3D NPB model has global weak solutions in the presence of boundaries with prescribed physical boundary conditions. Also, the long-time dynamics of solutions to the NPB system without any assumptions on $\overline{c_i(0)}$ is an open problem.

\section*{Acknowledgement}
R.H. was partially supported by a grant from the Simons Foundation (MP-TSM-00002783). Q.L. was partially supported by the AMS-Simons travel grant.

\bibliographystyle{plain}

\bibliography{reference}

\begin{thebibliography}{10}

\bibitem{abdo2021long}
Elie Abdo and Mihaela Ignatova.
\newblock Long time finite dimensionality in charged fluids.
\newblock {\em Nonlinearity}, 34(9):6173, 2021.

\bibitem{abdo2022space}
Elie Abdo and Mihaela Ignatova.
\newblock On the space analyticity of the {Nernst-Planck-Navier-Stokes} system.
\newblock {\em Journal of Mathematical Fluid Mechanics}, 24(2):51, 2022.

\bibitem{abdo2024long}
Elie Abdo and Mihaela Ignatova.
\newblock Long time dynamics of {N}ernst-{P}lanck-{N}avier-{S}tokes systems.
\newblock {\em Journal of Differential Equations}, 379:794--828, 2024.

\bibitem{abdo2024logarithmic}
Elie Abdo and Fizay-Noah Lee.
\newblock Logarithmic sobolev inequalities for bounded domains and applications
  to drift-diffusion equations.
\newblock {\em arXiv preprint arXiv:2402.18572}, 2024.

\bibitem{abdo2022global}
Elie Abdo, Fizay-Noah Lee, and Weinan Wang.
\newblock On the global well-posedness and gevrey regularity of some
  electrodiffusion models.
\newblock {\em arXiv preprint arXiv:2211.07686}, 2022.

\bibitem{alkhad2022electrochemical}
Mohammad~A. Alkhadra, Xiao Su, Matthew~E. Suss, Huanhuan Tian, Eric~N. Guyes,
  Amit~N. Shocron, Kameron~M. Conforti, J.~Pedro de~Souza, Nayeong Kim, Michele
  Tedesco, Khoiruddin Khoiruddin, I~Gede Wenten, Juan~G. Santiago, T.~Alan
  Hatton, and Martin~Z. Bazant.
\newblock Electrochemical methods for water purification, ion separations, and
  energy conversion.
\newblock {\em Chemical Reviews}, 122(16):13547--13635, 2022.

\bibitem{biler1994debye}
Piotr Biler, Waldemar Hebisch, and Tadeusz Nadzieja.
\newblock The {D}ebye system: existence and large time behavior of solutions.
\newblock {\em Nonlinear Analysis: Theory, Methods \& Applications},
  23(9):1189--1209, 1994.

\bibitem{bothe2014global}
Dieter Bothe, Andr{\'e} Fischer, and Jurgen Saal.
\newblock Global well-posedness and stability of electrokinetic flows.
\newblock {\em SIAM journal on Mathematical Analysis}, 46(2):1263--1316, 2014.

\bibitem{cole1965electrodiffusion}
Kenneth~S Cole.
\newblock Electrodiffusion models for the membrane of squid giant axon.
\newblock {\em Physiological Reviews}, 45(2):340--379, 1965.

\bibitem{constantin2019nernst}
Peter Constantin and Mihaela Ignatova.
\newblock On the {Nernst-Planck-Navier-Stokes} system.
\newblock {\em Archive for Rational Mechanics and Analysis}, 232:1379--1428,
  2019.

\bibitem{constantin2021interior}
Peter Constantin, Mihaela Ignatova, and Fizay-Noah Lee.
\newblock Interior electroneutrality in {Nernst-Planck-Navier-Stokes} systems.
\newblock {\em Archive for Rational Mechanics and Analysis}, 242(2):1091--1118,
  2021.

\bibitem{constantin2021nernst}
Peter Constantin, Mihaela Ignatova, and Fizay-Noah Lee.
\newblock {Nernst-Planck-Navier-Stokes} systems far from equilibrium.
\newblock {\em Archive for Rational Mechanics and Analysis}, 240:1147--1168,
  2021.

\bibitem{constantin2022existence}
Peter Constantin, Mihaela Ignatova, and Fizay-Noah Lee.
\newblock Existence and stability of nonequilibrium steady states of
  {Nernst-Planck-Navier-Stokes} systems.
\newblock {\em Physica D: Nonlinear Phenomena}, 442:133536, 2022.

\bibitem{constantin2022nernst}
Peter Constantin, Mihaela Ignatova, and Fizay-Noah Lee.
\newblock {Nernst-Planck-Navier-Stokes} systems near equilibrium.
\newblock {\em Pure and Applied Functional Analysis}, 7:175--196, 2022.

\bibitem{fischer2017global}
Andr{\'e} Fischer and J{\"u}rgen Saal.
\newblock Global weak solutions in three space dimensions for electrokinetic
  flow processes.
\newblock {\em Journal of Evolution Equations}, 17(1):309--333, 2017.

\bibitem{gajewski1986basic}
Herbert Gajewski and Konrad Gr{\"o}ger.
\newblock On the basic equations for carrier transport in semiconductors.
\newblock {\em Journal of Mathematical Analysis and Applications},
  113(1):12--35, 1986.

\bibitem{gao2014high}
Jun Gao, Wei Guo, Dan Feng, Huanting Wang, Dongyuan Zhao, and Lei Jiang.
\newblock High-performance ionic diode membrane for salinity gradient power
  generation.
\newblock {\em Journal of the American Chemical Society}, 136(35):12265--12272,
  2014.

\bibitem{gray1976validity}
Donald~D Gray and Aldo Giorgini.
\newblock The validity of the {Boussinesq} approximation for liquids and gases.
\newblock {\em International Journal of Heat and Mass Transfer},
  19(5):545--551, 1976.

\bibitem{ignatova2021global}
Mihaela Ignatova and Jingyang Shu.
\newblock Global solutions of the {Nernst-Planck-Euler} equations.
\newblock {\em SIAM journal on Mathematical Analysis}, 53(5):5507--5547, 2021.

\bibitem{ignatova2022global}
Mihaela Ignatova and Jingyang Shu.
\newblock Global smooth solutions of the {Nernst-Planck-Darcy} system.
\newblock {\em Journal of Mathematical Fluid Mechanics}, 24(1):26, 2022.

\bibitem{jasielec2021electrodiffusion}
Jerzy~J Jasielec.
\newblock Electrodiffusion phenomena in neuroscience and the
  {Nernst-Planck-Poisson} equations.
\newblock {\em Electrochem}, 2(2):197--215, 2021.

\bibitem{jerome2009global}
Joseph~W Jerome and Riccardo Sacco.
\newblock Global weak solutions for an incompressible charged fluid with
  multi-scale couplings: initial--boundary-value problem.
\newblock {\em Nonlinear Analysis: Theory, Methods \& Applications},
  71(12):e2487--e2497, 2009.

\bibitem{koch2004biophysics}
Christof Koch.
\newblock {\em Biophysics of computation: information processing in single
  neurons}.
\newblock Oxford University Press, 2004.

\bibitem{lee2016membrane}
Anna Lee, Jeffrey~W Elam, and Seth~B Darling.
\newblock Membrane materials for water purification: design, development, and
  application.
\newblock {\em Environmental Science: Water Research \& Technology},
  2(1):17--42, 2016.

\bibitem{lee2018diffusiophoretic}
Hyomin Lee, Junsuk Kim, Jina Yang, Sang~Woo Seo, and Sung~Jae Kim.
\newblock Diffusiophoretic exclusion of colloidal particles for continuous
  water purification.
\newblock {\em Lab on a Chip}, 18(12):1713--1724, 2018.

\bibitem{liu2020global}
Jian-Guo Liu and Jinhuan Wang.
\newblock Global existence for {Nernst-Planck-Navier-Stokes} system in ${R}^n$.
\newblock {\em Communications in Mathematical Sciences}, 18(6):1743--1754,
  2020.

\bibitem{lopreore2008computational}
Courtney~L Lopreore, Thomas~M Bartol, Jay~S Coggan, Daniel~X Keller, Gina~E
  Sosinsky, Mark~H Ellisman, and Terrence~J Sejnowski.
\newblock Computational modeling of three-dimensional electrodiffusion in
  biological systems: application to the node of ranvier.
\newblock {\em Biophysical Journal}, 95(6):2624--2635, 2008.

\bibitem{mayeli2021buoyancy}
Peyman Mayeli and Gregory~J Sheard.
\newblock Buoyancy-driven flows beyond the {Boussinesq} approximation: A brief
  review.
\newblock {\em International Communications in Heat and Mass Transfer},
  125:105316, 2021.

\bibitem{mock1983analysis}
Michael~Stephen Mock.
\newblock {\em Analysis of mathematical models of semiconductor devices}.
\newblock Boole Press, 1983.

\bibitem{mori2009numerical}
Yoichiro Mori and Charles Peskin.
\newblock A numerical method for cellular electrophysiology based on the
  electrodiffusion equations with internal boundary conditions at membranes.
\newblock {\em Communications in Applied Mathematics and Computational
  Science}, 4(1):85--134, 2009.

\bibitem{nicholson2000diffusion}
Charles Nicholson, Kevin~C Chen, Sabina Hrab{\v{e}}tov{\'a}, and Lian Tao.
\newblock Diffusion of molecules in brain extracellular space: theory and
  experiment.
\newblock {\em Progress in Brain Research}, 125:129--154, 2000.

\bibitem{petcu2009some}
Madalina Petcu, Roger~M Temam, and Mohammed Ziane.
\newblock Some mathematical problems in geophysical fluid dynamics.
\newblock {\em Handbook of Numerical Analysis}, 14:577--750, 2009.

\bibitem{pods2013electrodiffusion}
Jurgis Pods, Johannes Sch{\"o}nke, and Peter Bastian.
\newblock Electrodiffusion models of neurons and extracellular space using the
  {Poisson-Nernst-Planck} equations--numerical simulation of the intra-and
  extracellular potential for an axon model.
\newblock {\em Biophysical Journal}, 105(1):242--254, 2013.

\bibitem{qian1989electro}
Ning Qian and TJ~Sejnowski.
\newblock An electro-diffusion model for computing membrane potentials and
  ionic concentrations in branching dendrites, spines and axons.
\newblock {\em Biological Cybernetics}, 62(1):1--15, 1989.

\bibitem{rubinstein1990electro}
Isaak Rubinstein.
\newblock {\em Electro-diffusion of ions}.
\newblock SIAM, 1990.

\bibitem{rubinstein2000electro}
Isaak Rubinstein and Boris Zaltzman.
\newblock Electro-osmotically induced convection at a permselective membrane.
\newblock {\em Physical Review E}, 62(2):2238, 2000.

\bibitem{rubinstein2008direct}
Shmuel~M Rubinstein, Gor Manukyan, Adrian Staicu, Issac Rubinstein, Boris
  Zaltzman, Rob~GH Lammertink, Frieder Mugele, and Matthias Wessling.
\newblock Direct observation of a nonequilibrium electro-osmotic instability.
\newblock {\em Physical Review Letters}, 101(23):236101, 2008.

\bibitem{ryham2009existence}
Rolf~J Ryham.
\newblock Existence, uniqueness, regularity and long-term behavior for
  dissipative systems modeling electrohydrodynamics.
\newblock {\em arXiv preprint arXiv:0910.4973}, 2009.

\bibitem{savtchenko2017electrodiffusion}
Leonid~P Savtchenko, Mu~Ming Poo, and Dmitri~A Rusakov.
\newblock Electrodiffusion phenomena in neuroscience: a neglected companion.
\newblock {\em Nature Reviews Neuroscience}, 18(10):598--612, 2017.

\bibitem{schmuck2009analysis}
Markus Schmuck.
\newblock Analysis of the {Navier-Stokes-Nernst-Planck-Poisson} system.
\newblock {\em Mathematical Models and Methods in Applied Sciences},
  19(06):993--1014, 2009.

\bibitem{shen2022stability}
Rong Shen and Yong Wang.
\newblock Stability of the nonconstant stationary solution to the damped
  {Poisson-Nernst-Planck-Euler} equations.
\newblock {\em Mathematical Methods in the Applied Sciences},
  45(16):10331--10346, 2022.

\bibitem{tan2016computational}
Jinwang Tan and Emily~M Ryan.
\newblock Computational study of electro-convection effects on dendrite growth
  in batteries.
\newblock {\em Journal of Power Sources}, 323:67--77, 2016.

\bibitem{Tsybakov2008}
Alexandre~B Tsybakov.
\newblock {\em Introduction to nonparametric estimation}.
\newblock Springer Series in Statistics, Springer New York, 2008.

\bibitem{yang2010numerically}
JianWen Yang, ZuoHai Feng, XianRong Luo, and YuanRong Chen.
\newblock Numerically quantifying the relative importance of topography and
  buoyancy in driving groundwater flow.
\newblock {\em Science in China Series D: Earth Sciences}, 53(1):64--71, 2010.

\bibitem{yang2019review}
Zi~Yang, Yi~Zhou, Zhiyuan Feng, Xiaobo Rui, Tong Zhang, and Zhien Zhang.
\newblock A review on reverse osmosis and nanofiltration membranes for water
  purification.
\newblock {\em Polymers}, 11(8):1252, 2019.

\bibitem{zaltzman2007electro}
Boris Zaltzman and Isaak Rubinstein.
\newblock Electro-osmotic slip and electroconvective instability.
\newblock {\em Journal of Fluid Mechanics}, 579:173--226, 2007.

\bibitem{zhang2015global}
Zeng Zhang and Zhaoyang Yin.
\newblock Global well-posedness for the {Euler-Nernst-Planck-Poisson} system in
  dimension two.
\newblock {\em Nonlinear Analysis}, 125:30--53, 2015.

\bibitem{zhang2020inviscid}
Zeng Zhang and Zhaoyang Yin.
\newblock The inviscid limit and well-posedness for the
  {Euler-Nernst-Planck-Poisson} system.
\newblock {\em Applicable Analysis}, 99(2):181--213, 2020.

\bibitem{zhu2020ion}
Haitao Zhu, Bo~Yang, Congjie Gao, and Yaqin Wu.
\newblock Ion transfer modeling based on {N}ernst--{P}lanck theory for saline
  water desalination during electrodialysis process.
\newblock {\em Asia-Pacific Journal of Chemical Engineering}, 15(2):e2410,
  2020.

\end{thebibliography}

\appendix
\section{Global Regularity of the Mollified NPB System}\label{sec:appendix-a}

In this appendix we prove Proposition \ref{thm:mollified-sys} on the global well-posedness of the mollified system \eqref{sys:npb-full-mo}.
We fix some $\eta>0$, and drop the superscript $\eta$ when there is no confusion. As for the initial data, we keep the superscript $\eta$ to distinguish between the mollified and unmollified initial data.

\subsection{An iteration scheme}
For each $n\geq 0$, we consider the following iterative system
\noeqref{eqn:npb-full-mo-iteration-1}
\noeqref{eqn:npb-full-mo-iteration-2}
\noeqref{eqn:npb-full-mo-iteration-3}
\noeqref{eqn:npb-full-mo-iteration-4}
\noeqref{eqn:npb-full-mo-iteration-5}
\begin{subequations}\label{sys:npb-full-mo-iteration}
    \begin{align}
        &\partial_t c_i^{(n+1)} + \mathcal J_\eta   u^{(n+1)}\cdot \nabla c_i^{(n+1)}  - D\Delta c_i^{(n+1)} - D\frac{e}{k_B}\nabla\cdot (z_i\frac{1}{T^{(n+1)}} c_i^{(n+1)} \nabla \mathcal J_\eta \mathcal J_\eta \Psi^{(n)})= 0,\label{eqn:npb-full-mo-iteration-1}
        \\
        &-\varepsilon\Delta \Psi^{(n+1)} = \rho^{(n+1)} = \sum\limits_{i=1}^N e N_A z_i c_i^{(n+1)}, \label{eqn:npb-full-mo-iteration-2}
        \\
        &\partial_t u^{(n+1)} + \mathcal J_\eta u^{(n)}\cdot \nabla u^{(n+1)} - \nu\Delta u^{(n+1)} + \nabla p^{(n+1)}
        \\ 
        &\hspace{2cm}= g\left(\alpha_T(T^{(n+1)}-T_r) - \alpha_S\left(\sum\limits_{i=1}^N c^{(n)}_i M_i-S_r\right)\right) \vec{k} - \mathcal J_\eta   \left( \rho^{(n)} \nabla \mathcal J_\eta \mathcal J_\eta \Psi^{(n)} \right), \label{eqn:npb-full-mo-iteration-3}
        \\
        &\nabla\cdot u^{(n+1)} = 0,\label{eqn:npb-full-mo-iteration-4} 
        \\
        &\partial_t T^{(n+1)} + \mathcal J_\eta   u^{(n)} \cdot \nabla T^{(n+1)} - \kappa\Delta T^{(n+1)} = 0,\label{eqn:npb-full-mo-iteration-5}
    \end{align}
\end{subequations}
with initial conditions $(c^\eta_i(0), u^\eta_0, T_0^\eta)$. The constants $T_r= \overline{T_0^\eta}$ and $S_r = \sum\limits_{i=1}^N \overline{c^\eta_i(0)} M_i$ are fixed and independent of the spatial-time variables and the index $n$. In addition, we set
\[
 c_i^{(0)} = c^\eta_i(0), \quad u^{(0)}= u^\eta_0, \quad T^{(0)} = T_0^\eta, \quad -\varepsilon\Delta \Psi^{(0)} = \rho^{(0)} = \sum\limits_{i=1}^N e N_Az_i c_i^{(0)},
\]
which are all constants in time. 

For $n=0$, as $u^{(0)}$ is smooth and \eqref{eqn:npb-full-mo-iteration-5} is a linear transport-diffusion equation with periodic boundary conditions and smooth initial conditions, it has a unique smooth solution $T^{(1)}$ for any $t\geq 0$. Moreover, as $T_0^\eta\geq T^*$, we have $T^{(1)}\geq T^*.$
Next, since equation \eqref{eqn:npb-full-mo-iteration-3} is linear in $u^{(1)}$ and has smooth forcing terms depending on $T^{(1)}$ and $c_i^{(0)}$, and $u^{(1)}$ is transported by a smooth velocity field $\mathcal J_\eta u^{(0)}$, we infer the existence of a unique smooth solution $u^{(1)}$ for any $t\geq 0$. As $T^{(1)}\geq T^*$, one has that $\frac{1}{T^{(1)}}$ is smooth. Since equation \eqref{eqn:npb-full-mo-iteration-1} is linear in $c_i^{(1)}$ with all other functions in the equation being smooth, one can obtain a unique smooth solution $c_i^{(1)}$ for each $i=1,2,\dots,N$ and for each $t\geq 0$. In addition, following the argument in \cite{constantin2019nernst} and the fact that 
$\int_0^t\|\frac{1}{T^{(1)}}\|_{L^\infty}^2 \| \na\mathcal J_\eta \mathcal J_\eta \Psi^{(0)}\|_{L^\infty}^2 dt'$ 
is bounded for any $t\geq 0$, we deduce that $c_i^{(1)}\geq 0$. 

Having constructed smooth solutions $c_i^{(1)}$, $u^{(1)}$, $T^{(1)}$, $\rho^{(1)}$, and $\Psi^{(1)}$ for the iteration $n=0$,  we can proceed in the same manner and obtain another family of solutions $c_i^{(2)}$, $u^{(2)}$, $T^{(2)}$, $\rho^{(2)}$, and $\Psi^{(2)}$ corresponding to the iteration $n=1$. Repeating the same argument, one can obtain a sequence of smooth solutions $c_i^{(n)}$, $u^{(n)}$, $T^{(n)}$, $\rho^{(n)}$, and $\Psi^{(n)}$ for each $n\in \NN$. In addition, it holds that $T^{(n)}\geq T^*$ and $c_i^{(n)} \geq 0$ for all times $t \ge 0$ and all $n \in \NN$. As $\overline{T^{(n)}} = \overline{T_0^\eta}$ and $\overline{c_i^{(n)}} = \overline{c_i^\eta(0)}$ are invariant in time, one has $\overline{u^{(n)}} = \overline{u_0^\eta} = 0$.

\subsection{Energy estimates}
Having a sequence of solutions for $n\in \mathbb N$, we now establish the energy estimates that are needed to prove the global well-posedness and regularity of the mollified system. These estimates will be performed in several steps.

\smallskip

\noindent {\bf Step 1. The bound for $\|T^{(n+1)}\|_{L^2}$.} 
By integrating by parts and using the divergence-free condition of $\mathcal J_\eta  u^{(n)}$, the $L^2$ norm of the iterative temperature $T^{(n+1)}$ evolves according to 
$$
\frac{1}{2} \frac{d}{d t}\|T^{(n+1)}\|_{L^2}^2+\kappa\|\nabla T^{(n+1)}\|_{L^2}^2=0,
$$
which implies that for any $t \ge 0$,
\begin{align}\label{est:iteration-TL2}
    \|T^{(n+1)}(t)\|_{L^2}^2 + 2\kappa \int_0^t \|\nabla T^{(n+1)}(s)\|_{L^2}^2 ds = \|T^\eta_0\|_{L^2}^2.
\end{align}

\noindent {\bf Step 2. The bound for $\|c_i^{(n+1)}\|_{L^2}$.}
The $L^2$ norm of each concentration $c_i^{(n+1)}$ satisfies
\[
\frac{1}{2} \frac{d}{d t}\|c_i^{(n+1)}\|_{L^2}^2 + D \|\nabla c_i^{(n+1)}\|_{L^2}^2= -D\frac{e}{k_B} \int_{\mathbb{T}^3} z_i\frac{1}{T^{(n+1)}} c_i^{(n+1)} \nabla c_i^{(n+1)} \nabla \mathcal J_\eta \mathcal J_\eta \Psi^{(n)}  dx,
\] due to the divergence-free condition obeyed by $\mathcal J_\eta   u^{(n+1)}$. 
Applications of H\"older's, Young's, and the Sobolev inequalities give rise to the energy inequality
\begin{align*}
   \frac{1}{2} \frac{d}{d t}\|c_i^{(n+1)}\|_{L^2}^2 + \frac{1}{2}D \|\nabla c_i^{(n+1)}\|_{L^2}^2 \leq &C \left\|\frac{1}{T^{(n+1)}}\right\|_{L^\infty}^2 \|\nabla \mathcal J_\eta \mathcal J_\eta \Psi^{(n)} \|_{L^\infty}^2  \|c_i^{(n+1)}\|_{L^2}^2
   \\
   \leq & C \left(\frac{1}{T^*}\right)^{-2} \|\mathcal J_\eta  \mathcal J_\eta  \Psi^{(n)}\|_{H^3}^2 \|c_i^{(n+1)}\|_{L^2}^2.
\end{align*}
Using the smoothing  properties of the mollifiers and employing a standard duality argument, one has
\begin{align}\label{est:iteration-duality}
    \|\mathcal J_\eta  \mathcal J_\eta  \Psi^{(n)}\|_{H^3} \leq C \eta^{-3} \| \rho^{(n)}\|_{H^{-2}}\leq C \eta^{-3} \| \rho^{(n)}\|_{L^1} \leq C\eta^{-3} \sum\limits_{i=1}^N \|c_i^{(n)}\|_{L^1} = C\eta^{-3} \sum\limits_{i=1}^N \overline{c_i^\eta(0)},
\end{align}
 where the last equality follows from the nonnegativity of the ionic concentrations $c_i^{(n+1)}$ and the fact that their spacial means are invariant in time. Consequently, one can infer that
\begin{align*}
   \frac{1}{2} \frac{d}{d t}\|c_i^{(n+1)}\|_{L^2}^2 + \frac{1}{2}D \|\nabla c_i^{(n+1)}\|_{L^2}^2 \leq  C \eta^{-6} \left(\frac{1}{T^*}\right)^{-2} \sum\limits_{i=1}^N \overline{c_i^\eta(0)}^2 \|c_i^{(n+1)}\|_{L^2}^2 =: \widetilde{C}\|c_i^{(n+1)}\|_{L^2}^2,
\end{align*}
where $\widetilde{C}$ is a constant independent of space and time. By Gr\"onwall's inequality, for any $t\geq 0$ we have
\begin{align}\label{est:iteration-ciL2}
    \|c_i^{(n+1)}(t)\|_{L^2}^2 + \int_0^t D\|\nabla c_i^{(n+1)}(s)\|_{L^2}^2 e^{\widetilde{C}(t-s)} ds \leq \|c_i^\eta(0)\|_{L^2}^2 e^{\widetilde{C}t}.
\end{align} 


\noindent {\bf Step 3. The bound for $\|u^{(n+1)}\|_{L^2}$.} 
The $L^2$ norm of the iterative velocity $u^{(n+1)}$ satisfies
\begin{align*}
    \frac{1}{2} \frac{d}{d t}\|u^{(n+1)}\|_{L^2}^2+\nu\|\nabla u^{(n+1)}\|_{L^2}^2=&-\int_{\mathbb{T}^3} \mathcal J_\eta  \left(\rho^{(n)} \nabla\mathcal J_\eta \mathcal J_\eta \Psi^{(n)} \right)\cdot u^{(n+1)} dx 
    \\
    &+ \int_{\mathbb{T}^3} \left(g\alpha_T (T^{(n+1)}-T_r) u^{(n+1)}_3 -g\alpha_S(\sum\limits_{i=1}^N c_i^{(n)} M_i - S_r) u^{(n+1)}_3  \right)dx.
\end{align*}
By making use of the mean-free property obeyed by $u^{(n+1)}$, and applying H\"older's, Young's, the Poincar\'e and the Sobolev inequalities, we have
\begin{align*}
    &\frac{1}{2} \frac{d}{d t}\|u^{(n+1)}\|_{L^2}^2+\nu\|\nabla u^{(n+1)}\|_{L^2}^2
    \\
    \leq &C \Big(\sum\limits_{i=1}^N \|c_i^{(n)}\|_{L^2} \|\nabla \mathcal J_\eta \mathcal J_\eta \Psi^{(n)} \|_{L^\infty} + \|T^{(n+1)}\|_{L^2} + \sum\limits_{i=1}^N \|c_i^{(n)}\|_{L^2} \Big) \|u^{(n+1)}\|_{L^2}
    \\
    \leq &C(\sum\limits_{i=1}^N \|c_i^{(n)}\|_{L^2}^2 + \|T^{(n+1)}\|_{L^2}^2) + \frac\nu2 \|\na u^{(n+1)}\|_{L^2}^2,
\end{align*}
where the second inequality follows from the uniform-in-time bound \eqref{est:iteration-duality} for the mollified iterative potential. As $c_i^{(n)}$ obeys the same estimate derived for $c_i^{(n+1)}$ in \eqref{est:iteration-ciL2}, and thanks to the uniform-in-time boundedness of the iterative temperatures in $L^2$ described by \eqref{est:iteration-TL2}, the latter differential inequality boils down to
\begin{align*}
    \frac{d}{d t}\|u^{(n+1)}\|_{L^2}^2+\nu\|\nabla u^{(n+1)}\|_{L^2}^2\leq C(\|c_i^\eta(0)\|_{L^2}^2 e^{\widetilde{C}t} + \|T_0^\eta\|_{L^2}^2).
\end{align*}
An application of Gr\"onwall's inequality yields 
\begin{align}\label{est:iteration-uL2}
    \|u^{(n+1)}(t)\|_{L^2}^2 + \nu \int_0^t \|\nabla u^{(n+1)}(s)\|_{L^2}^2 ds \leq \|u_0^\eta\|_{L^2}^2 + C(\|c_i^\eta(0)\|_{L^2}^2 e^{\widetilde{C}t} + \|T_0^\eta\|_{L^2}^2 t).
\end{align} 

\smallskip

\noindent {\bf Step 4. Bounds for the time derivatives.} 
In order to obtain good controls of $\partial_t T^{(n+1)}$, $\partial_t c_i^{(n+1)}$, and $\partial_t u^{(n+1)}$ in appropriate Lebesgue spaces, we take the inner product of equations \eqref{eqn:npb-full-mo-iteration-1} and \eqref{eqn:npb-full-mo-iteration-5} with a scalar test function $\phi\in L^2(0,\mathcal T;H^1)$, and \eqref{eqn:npb-full-mo-iteration-3} with a vector-valued test function $\varphi\in L^2(0,\mathcal T;V)$. By H\"older's inequality and the Sobolev inequality,  we infer that 
\begin{align*}
    \left|\langle \mathcal J_\eta   u^{(n+1)} \cdot\na c_i^{(n+1)}, \phi\rangle\right| \leq C\|\mathcal J_\eta  u^{(n+1)}\|_{L^3} \|\na c_i^{(n+1)}\|_{L^2} \|\phi\|_{L^6} \leq C_\eta \|u^{(n+1)}\|_{L^2}  \|\na c_i^{(n+1)}\|_{L^2} \|\phi\|_{H^1}.
\end{align*}
Thanks to the uniform-in-time bounds \eqref{est:iteration-ciL2} and \eqref{est:iteration-uL2}, it follows that $\mathcal J_\eta  u^{(n+1)} \cdot\na c_i^{(n+1)}\in L^{2}(0,\mathcal T;H^{-1})$. Similarly, we can show that $\mathcal J_\eta u^{(n)} \cdot\na u^{(n+1)}\in L^{2}(0,\mathcal T;V')$ and $\mathcal J_\eta  u^{(n)} \cdot\na T^{(n+1)}\in L^{2}(0,\mathcal T;H^{-1})$. As for the electromigration term, we integrate by parts and use the potential bound \eqref{est:iteration-duality} to estimate 
\begin{align*}
   \left|\left\langle \frac{1}{T^{(n+1)}} c_i^{(n+1)} \nabla \mathcal J_\eta \mathcal J_\eta \Psi^{(n)}, \na\phi\right\rangle\right| \leq &C \left\|\frac{1}{T^{(n+1)}}\right\|_{L^\infty} \|c_i^{(n+1)}\|_{L^2} \|\nabla \mathcal J_\eta \mathcal J_\eta \Psi^{(n)}\|_{L^\infty} \|\na\phi\|_{L^2}
   \\
   \leq &C\frac{1}{T^*} \eta^{-3} \sum\limits_{i=1}^N \overline{c_i^\eta(0)} \|c_i^{(n+1)}\|_{L^2}\|\phi\|_{H^1}.
\end{align*}
Consequently, we conclude that  $\na\cdot\left(\frac{1}{T^{(n+1)}} c_i^{(n+1)} \nabla \mathcal J_\eta \mathcal J_\eta \Psi^{(n)}\right)\in L^\infty(0,\mathcal T;H^{-1})$. The remaining terms can be estimated in a similar fashion. As a result, one has
\begin{align}\label{bound:time-de}
    &\{\partial_t c_i^{(n)},\partial_t T^{(n)}\}\,\, \text{are uniformly bounded in} \,\, L^{2}(0,\mathcal T;H^{-1}), 
    \\
    &\{\partial_t u^{(n)}\} \,\, \text{are uniformly bounded in} \,\, L^{2}(0,\mathcal T; V').
\end{align}

\subsection{Existence of weak solutions}\label{subsec:weak-sol}
The proof of the existence of weak solutions follows similarly to the one in Section \ref{sec:existence}, and is indeed much simpler since the sequence of solutions (in $n$) to the iteration scheme \eqref{sys:npb-full-mo-iteration} have better estimates compared to the sequence of solutions (in $\eta_k$) to the mollified system \eqref{sys:npb-full-mo}. Therefore, we omit the details here.

\subsection{Regularity of solutions} 
We fix $\eta>0$ and again drop the superscript $\eta$ except for the initial conditions. For any $s\in\mathbb N$, consider a multi-index $\alpha\in \mathbb N^3$ such that $|\alpha|\leq s$. Then the $H^s$ evolution of $T$ reads
\begin{align*}
    \frac{d}{dt}\|T\|_{H^s}^2 + 2\kappa \|\nabla T\|_{H^s}^2 = - \sum\limits_{0\leq |\alpha|\leq s} \langle D^\alpha (\mathcal J_\eta  u \cdot \nabla T), D^\alpha T\rangle.
\end{align*}
As $\mathcal J_\eta  u$ is divergence-free, we have $\langle  \mathcal J_\eta  u \cdot \nabla D^\alpha T, D^\alpha T\rangle = 0$.  We use the H\"older and the Sobolev inequalities to  bound 
\begin{equation}\label{est-mollifed-reg-1}
    \begin{split}
        &\left|\sum\limits_{0\leq |\alpha|\leq s} \langle D^\alpha (\mathcal J_\eta  u \cdot \nabla T), D^\alpha T\rangle\right| = \left| \sum\limits_{0\leq |\alpha|\leq s} \sum\limits_{0<\beta\leq \alpha} \langle D^\beta \mathcal J_\eta  u \cdot D^{\alpha-\beta} \nabla T, D^\alpha T\rangle\right|
   \\
   \leq & C \sum\limits_{0\leq |\alpha|\leq s} \sum\limits_{0<\beta\leq \alpha} \|D^\beta \mathcal J_\eta  u\|_{H^2} \|D^{\alpha-\beta} \nabla T\|_{L^2} \|D^\alpha T\|_{L^2}
   \leq  C_{s,\eta} \|u\|_{L^2} \|T\|_{H^s}^2.
    \end{split}
\end{equation}
Thus it follows that
\begin{align*}
    \frac{d}{dt}\|T\|_{H^s}^2 + 2\kappa \|\nabla T\|_{H^s}^2 \leq  C_{s,\eta} \|u\|_{L^2} \|T\|_{H^s}^2.
\end{align*}
Application of Gr\"onwall's inequality and the regularity criterion $u\in C([0,\mathcal T]; L^2)$ (that holds for any $\mathcal T>0$) yields
\begin{equation}\label{est-mollifed-reg-2}
    \|T(t)\|_{H^s}^2 + 2\kappa \int_0^{\mathcal T} \|\nabla T(s)\|_{H^s}^2 ds \leq C_{s,\eta,\mathcal T, \|u_0\|_{L^2}, \|T_0\|_{L^2}},
\end{equation}
for any $t\in[0,\mathcal T]$. Here we have used the fact that $\|u_0^\eta\|_{L^2} \leq \|u_0\|_{L^2}$ and $\|T_0^\eta\|_{H^s} \leq C_{s,\eta}\|T_0\|_{L^2}$. Typically, as the initial condition $T_0^\eta \in C^\infty(\mathbb T^3)$, one can perform the estimate above for any $s\in \mathbb N$ and eventually conclude that
\begin{equation}\label{mollfied-T-reg}
    T \in C([0,\mathcal T]; C^\infty(\mathbb T^3)).
\end{equation}

Next, we address the time evolution of  the $H^s$ norms of the concentrations $c_i$. These are described by
\begin{align*}
    \frac{d}{dt}\|c_i\|_{H^s}^2 + 2D \|\nabla c_i\|_{H^s}^2 = - \sum\limits_{0\leq |\alpha|\leq s} \langle D^\alpha (\mathcal J_\eta  u \cdot \nabla c_i), D^\alpha c_i\rangle - D\frac{ez_i}{k_B} \sum\limits_{0\leq |\alpha|\leq s} \langle D^\alpha (\frac{1}{T} c_i \nabla \mathcal J_\eta  \mathcal J_\eta  \Psi), D^\alpha \nabla c_i\rangle.
\end{align*}
The first nonlinear term on the right-hand side can be estimated similarly as in \eqref{est-mollifed-reg-1}, and one can get
\begin{align*}
    \left|\sum\limits_{0\leq |\alpha|\leq s} \langle D^\alpha (\mathcal J_\eta  u \cdot \nabla c_i), D^\alpha c_i\rangle\right| \leq C_{s,\eta} \|u\|_{L^2} \|c_i\|_{H^s}^2.
\end{align*}
For any $m\in\mathbb N$ and $|\alpha|=m$, the derivative $D^\alpha (\frac1T)$ consists of terms whose denominators are $T^{m'}$ with $m'\leq m+1$, and whose numerators are the product of derivatives of $T$ with order at most $m$. Thanks to the fact that $T\geq T^*$ and \eqref{est-mollifed-reg-2}, one can conclude that
\[
  \left\|D^\alpha (\frac1T)\right\|_{L^\infty} \leq C \left\| \frac1T\right\|_{H^{s+2}} \leq C_{s,\eta,\mathcal T, \|u_0\|_{L^2}, \|T_0\|_{L^2}, T^*}.
\]
By H\"older's inequality and Young's inequality, and using an estimate similar to \eqref{est:iteration-duality} with $3$ replaced by $s+2$, we have
\begin{align*}
    \left| D\frac{ez_i}{k_B} \sum\limits_{0\leq |\alpha|\leq s} \langle D^\alpha (\frac{1}{T} c_i \nabla \mathcal J_\eta  \mathcal J_\eta  \Psi), D^\alpha \nabla c_i\rangle \right| \leq D \|\nabla c_i\|_{H^s}^2 + C \|c_i\|_{H^s}^2, 
\end{align*}
where the constant in the last step depends on $s$, $\eta$, $\|u_0\|_{L^2}$, $\|T_0\|_{L^2}$, and $\overline{c_i(0)}$. Hence, the Sobolev $H^s$ norm of $c_i$ obeys 
\begin{align*}
   \frac{d}{dt}\|c_i\|_{H^s}^2 + D \|\nabla c_i\|_{H^s}^2  \leq C \|c_i\|_{H^s}^2,
\end{align*}
which, together with Gr\"onwall's inequality, implies that
\begin{align}\label{est-mollifed-reg-3}
    \|c_i(t)\|_{H^s}^2 + D \int_0^{\mathcal T} \|\nabla c_i(s)\|_{H^s}^2 ds \leq C,
\end{align} for any $t \in [0, \mathcal{T}]$.
As $c^\eta_i(0) \in C^\infty(\mathbb T^3)$, the above is true for any $s\in \mathbb N$, and therefore,
\begin{equation}\label{mollfied-c-reg}
    c_i \in C([0,\mathcal T]; C^\infty(\mathbb T^3)), \quad \text{for } \, i=1,2,\dots,N.
\end{equation}

Finally, the $H^s$ evolution of $u$ reads
\begin{equation} \label{est-mollifed-reg-4}
    \begin{split}
        &\frac{d}{dt}\|u\|_{H^s}^2 + 2\nu \|\nabla u\|_{H^s}^2 
     \\
     = &- \sum\limits_{0\leq |\alpha|\leq s} \langle D^\alpha (\mathcal J_\eta  u \cdot \nabla u), D^\alpha u\rangle  - \sum\limits_{0\leq |\alpha|\leq s} g\alpha_S \left\langle D^\alpha \left(\sum\limits_{i=1}^N M_i( c_i - \overline{c_i^\eta(0)}) \right), D^\alpha u_3 \right\rangle
     \\
     &+ \sum\limits_{0\leq |\alpha|\leq s}  g\alpha_T \langle D^\alpha\left( T-\overline{T_0^\eta}\right), D^\alpha u_3 \rangle - \sum\limits_{0\leq |\alpha|\leq s} \langle D^\alpha\mathcal J_\eta \left(\rho\nabla\mathcal J_\eta  \mathcal J_\eta  \Psi\right), D^\alpha u \rangle. 
    \end{split}
\end{equation}
The first nonlinear term on the right-hand side can be estimated similarly as in \eqref{est-mollifed-reg-1}, and one obtains
\begin{align*}
    \left|\sum\limits_{0\leq |\alpha|\leq s} \langle D^\alpha (\mathcal J_\eta  u \cdot \nabla u), D^\alpha u\rangle\right| \leq C_{s,\eta} \|u\|_{L^2} \|u\|_{H^s}^2.
\end{align*}
By H\"older's and Young's inequalities, we have
\begin{align*}
    &\left|\sum\limits_{0\leq |\alpha|\leq s} g\alpha_S \left\langle D^\alpha \left(\sum\limits_{i=1}^N M_i( c_i - \overline{c_i^\eta(0)}) \right), D^\alpha u_3 \right\rangle \right| + \left| \sum\limits_{0\leq |\alpha|\leq s}  g\alpha_T \langle D^\alpha\left( T-\overline{T_0^\eta}\right), D^\alpha u_3 \rangle\right| 
    \\
    \leq & C (1+ \sum\limits_{i=1}^N \|c_i\|_{H^s}^2 + \|T\|_{H^s}^2) + C\|u\|_{H^s}^2.
\end{align*}
Using in addition the bound \eqref{est:iteration-duality}, we estimate
\begin{align*}
    \left|\sum\limits_{0\leq |\alpha|\leq s} \langle D^\alpha\mathcal J_\eta \left(\rho\nabla\mathcal J_\eta  \mathcal J_\eta  \Psi\right), D^\alpha u \rangle  \right| \leq C \sum\limits_{i=1}^N \|c_i\|_{L^2}^2  + C\|u\|_{H^s}^2.
\end{align*}
Combining the estimates above for the right-hand side of \eqref{est-mollifed-reg-4}, one has
\begin{align*}
     \frac{d}{dt}\|u\|_{H^s}^2 + 2\nu \|\nabla u\|_{H^s}^2 \leq C (1+ \sum\limits_{i=1}^N \|c_i\|_{H^s}^2 + \|T\|_{H^s}^2) + C(1+\|u\|_{L^2}) \|u\|_{H^s}^2.
\end{align*}
By Gr\"onwall's inequality and thanks to \eqref{est-mollifed-reg-2} and \eqref{est-mollifed-reg-3}, we deduce that
\begin{align*}
    \|u(t)\|_{H^s}^2 + 2\nu \int_0^{\mathcal T} \|\nabla u(s)\|_{H^s}^2 ds \leq C,
\end{align*} for any $t \in [0,\mathcal{T}]$.
As $u_0^\eta \in C^\infty(\mathbb T^3)$, the above is true for any $s\in \mathbb N$, and therefore,
\begin{equation}\label{mollfied-u-reg}
    u \in C([0,\mathcal T]; C^\infty(\mathbb T^3)).
\end{equation}

\subsection{Uniqueness of solutions} Fix $\eta>0$ and let $(u^1, c_i^1, T^1)$ and $(u^2, c_i^2, T^2)$ be two smooth solutions to system \eqref{sys:npb-full-mo} with the same initial conditions $(u_0^\eta, c_i^\eta(0), T_0^\eta)$. Denote their difference by $(u,c_i,T)=(u^1-u^2, c_i^1-c_i^2, T^1-T^2)$, and let $\rho=\rho^1-\rho^2 := \sum\limits_{i=1}^N eN_Az_i(c_i^1 - c_i^2)$ and $p=p^1-p^2$. Then $(u,c_i,T)$ satisfies
\begin{align*}
     &\partial_t c_i + \mathcal J_\eta  u\cdot \nabla c_i^1 + \mathcal J_\eta  u^2\cdot \nabla c_i  - D\Delta c_i + D\frac{e}{k_B}\nabla\cdot (z_i\frac{T}{T^1 T^2} c_i^1 \nabla \mathcal J_\eta \mathcal J_\eta \Psi^1)
     \\
     &\hspace{2cm}- D\frac{e}{k_B}\nabla\cdot (z_i\frac{1}{T^2} c_i \nabla \mathcal J_\eta \mathcal J_\eta \Psi^1) - D\frac{e}{k_B}\nabla\cdot (z_i\frac{1}{T^2} c_i^2 \nabla \mathcal J_\eta \mathcal J_\eta \Psi) = 0,
        \\
        &-\varepsilon\Delta \Psi = \rho = \sum\limits_{i=1}^N eN_A z_i c_i, 
        \\
        &\partial_t u + \mathcal J_\eta u\cdot \nabla u^1 + \mathcal J_\eta u^2\cdot \nabla u - \nu\Delta u + \nabla p 
        \\
         &\hspace{2cm}= g\left(\alpha_T T - \alpha_S\sum\limits_{i=1}^N c_i M_i\right) \vec{k} -\mathcal J_\eta  \left( \rho \nabla \mathcal J_\eta \mathcal J_\eta \Psi^1\right)  -\mathcal J_\eta  \left( \rho^2 \nabla \mathcal J_\eta \mathcal J_\eta \Psi\right), 
        \\
        &\nabla\cdot u = 0,
        \\
        &\partial_t T +\mathcal J_\eta   u\cdot \nabla T^1 +\mathcal J_\eta   u^2\cdot \nabla T - \kappa\Delta T = 0.
\end{align*}
The $L^2$ time evolution of $(c_i,u,T)$ is given by 
\begin{align*}
    &\frac{1}{2}\frac{d}{dt}\left(\|u\|_{L^2}^2 + \sum\limits_{i=1}^N \|c_i\|_{L^2}^2 + \|T\|_{L^2}^2 \right) + \nu \|\nabla u\|_{L^2}^2 + D\sum\limits_{i=1}^N \|\nabla c_i\|_{L^2}^2 + \kappa \|\nabla T\|_{L^2}^2
    \\
    = &  \sum\limits_{i=1}^N \Bigg[-\left\langle \mathcal J_\eta  u\cdot \nabla c_i^1, c_i\right\rangle +  D\frac{e}{k_B} \left\langle  z_i\frac{T}{T^1 T^2} c_i^1 \nabla \mathcal J_\eta \mathcal J_\eta \Psi^1, \nabla c_i\right\rangle - D\frac{e}{k_B} \left\langle  z_i\frac{1}{T^2} c_i \nabla \mathcal J_\eta \mathcal J_\eta \Psi^1, \nabla c_i\right\rangle 
    \\
    &\hspace{0.5cm}-  D\frac{e}{k_B} \left\langle  z_i\frac{1}{T^2} c_i^2 \nabla \mathcal J_\eta \mathcal J_\eta \Psi, \nabla c_i\right\rangle \Bigg] - \left\langle \mathcal J_\eta u\cdot \nabla u^1, u  \right\rangle +  \left\langle g\left(\alpha_T T - \alpha_S\sum\limits_{i=1}^N c_i M_i\right) ,u_3  \right\rangle
    \\
    &- \left\langle \mathcal J_\eta  \left( \rho \nabla \mathcal J_\eta \mathcal J_\eta \Psi^1\right) , u \right\rangle - \left\langle \mathcal J_\eta  \left( \rho^2 \nabla \mathcal J_\eta \mathcal J_\eta \Psi\right) , u \right\rangle - \left\langle \mathcal J_\eta u\cdot \nabla T^1, T  \right\rangle.
\end{align*}
Thanks to H\"older's inequality, Young's inequality, and the Sobolev inequality, it follows that
\begin{align*}
    &\sum\limits_{i=1}^N\left| \left\langle \mathcal J_\eta  u\cdot \nabla c_i^1, c_i\right\rangle \right| + \left| \left\langle \mathcal J_\eta  u\cdot \nabla u^1, u\right\rangle \right| + \left| \left\langle \mathcal J_\eta  u\cdot \nabla T^1, T\right\rangle \right| + \left| \left\langle g\left(\alpha_T T - \alpha_S\sum\limits_{i=1}^N c_i M_i\right) ,u_3  \right\rangle \right|
    \\
    &\hspace{1cm}+ \left| \left\langle \mathcal J_\eta  \left( \rho \nabla \mathcal J_\eta \mathcal J_\eta \Psi^1\right) , u \right\rangle \right| + \left| \left\langle \mathcal J_\eta  \left( \rho^2 \nabla \mathcal J_\eta \mathcal J_\eta \Psi\right) , u \right\rangle \right|
    \\
    \leq &C \left(1+ \sum\limits_{i=1}^N (\|c_i^1\|_{H^3} + \|c_i^2\|_{H^3} ) + \|u^1\|_{H^3} +  \|T^1\|_{H^3}\right) \left(\|u\|_{L^2}^2 + \sum\limits_{i=1}^N \|c_i\|_{L^2}^2 + \|T\|_{L^2}^2 \right).
\end{align*}
As $T^1, T^2 \geq T^*$, and thanks to \eqref{est:iteration-duality}, we have
\begin{align*}
     &\sum\limits_{i=1}^ND\frac{e}{k_B} \Bigg[ \left|\left\langle  z_i\frac{T}{T^1 T^2} c_i^1 \nabla \mathcal J_\eta \mathcal J_\eta \Psi^1, \nabla c_i\right\rangle \right| + \left|\left\langle  z_i\frac{1}{T^2} c_i \nabla \mathcal J_\eta \mathcal J_\eta \Psi^1, \nabla c_i\right\rangle  \right| 
     \\
     &\hspace{2cm}+ \left|\left\langle  z_i\frac{1}{T^2} c_i^2 \nabla \mathcal J_\eta \mathcal J_\eta \Psi, \nabla c_i\right\rangle \right| \Bigg]
     \\
     \leq & C \left(1+ \sum\limits_{i=1}^N \|c_i^1\|_{H^3} + \sum\limits_{i=1}^N \|c_i^2\|_{H^3} \right) \left( \sum\limits_{i=1}^N \|c_i\|_{L^2}^2 + \|T\|_{L^2}^2 \right) + \frac12 D \sum\limits_{i=1}^N \|\nabla c_i\|_{L^2}^2.
\end{align*}
Thus, it follows that 
\begin{align*}
     &\frac{d}{dt}\left(\|u\|_{L^2}^2 + \sum\limits_{i=1}^N \|c_i\|_{L^2}^2 + \|T\|_{L^2}^2 \right) + \nu \|\nabla u\|_{L^2}^2 + D\sum\limits_{i=1}^N \|\nabla c_i\|_{L^2}^2 + \kappa \|\nabla T\|_{L^2}^2
    \\
    \leq &C \left(1+ \sum\limits_{i=1}^N \|c_i^1\|_{H^3} + \sum\limits_{i=1}^N \|c_i^2\|_{H^3} + \|u^1\|_{H^3} +  \|T^1\|_{H^3}\right) \left(\|u\|_{L^2}^2 + \sum\limits_{i=1}^N \|c_i\|_{L^2}^2 + \|T\|_{L^2}^2 \right).
\end{align*}
Due to Gr\"onwall's inequality and the regularity \eqref{mollfied-T-reg}, \eqref{mollfied-c-reg}, and \eqref{mollfied-u-reg} obeyed by both $(u^1,c_i^1,T^1)$ and $(u^2,c_i^2,T^2)$, we conclude that $u=c_i=T=0$, from which we deduce the uniqueness of solutions. This completes the proof of Proposition \ref{thm:mollified-sys}.

\end{document}